\documentclass[a4paper,11pt,oneside]{article}

\usepackage[margin=18mm,includefoot]{geometry}

\usepackage{color}
\definecolor{cornellrot}{RGB}{179,27,27}

\usepackage{hyperref}
\hypersetup{colorlinks=true,allcolors=cornellrot}

\usepackage{amsmath}
\usepackage{amssymb}
\usepackage{amsthm}
\usepackage{graphicx}
\usepackage{calrsfs}

\newcommand{\Z}{\mathbb{Z}}

\renewcommand{\geq}{\geqslant}
\renewcommand{\leq}{\leqslant}

\newcommand{\larw}[1]{\raisebox{-2pt}{$\xleftarrow{#1}$}}
\newcommand{\rarw}[1]{\raisebox{-2pt}{$\xrightarrow{#1}$}}

\newcommand{\brkt}[1]{\textup{(}#1\textup{)}}

\newcommand{\sset}   [1] {\{ {#1} \}}
\newcommand{\pres}   [2] {\langle\, {#1} \mid {#2} \,\rangle}
\newcommand{\gpres}  [1] {\langle\, {#1} \,\rangle}
\newcommand{\ngpres} [1] {\langle\!\langle\, {#1} \,\rangle\!\rangle}

\usepackage{scalerel}

\newcommand{\hugepres}[2]{
  \stretchrel{\Bigg\langle}{\left. {#1}\hspace{7pt}\middle|\hspace{2pt} {#2} \right.}
  \stretchrel*{\Bigg\rangle}{\left. {#1}\hspace{7pt}\middle|\hspace{2pt} {#2} \right.}
}

\newcommand{\ie}{i.e.~}

\newcommand{\egcomma}{e.g.,~}

\newcommand{\blanks}{{\mkern 1.5mu\cdot\mkern 1.5mu,\cdot\mkern 1.5mu}}

\makeatletter
\newcommand \mydots {\leavevmode \cleaders \hb@xt@ .33em{\hss .\hss }\hfill \kern \z@}
\makeatother

\newtheorem{theorem}              {Theorem}
\newtheorem{corollary}  [theorem] {Corollary}
\newtheorem{lemma}      [theorem] {Lemma}
\newtheorem{prop}       [theorem] {Proposition}
\newtheorem{question}   [theorem] {Question}

\newtheorem{maintheorem}{Theorem}
\newtheorem{maincorollary}[maintheorem]{Corollary}

\theoremstyle{definition}

\newtheorem*{convention*}{Notational convention}
\newtheorem*{acknowledgement*}{Acknowledgement}

\newtheorem{remark}[theorem]{Remark}

\newtheorem{excursion}{Excursion}

\numberwithin{theorem}{section}

\begin{document}

\title{A class of digraph groups\\defined by balanced presentations}

\author{Johannes Cuno and Gerald Williams\,\footnote{Both authors acknowledge the support of the London Mathematical Society, Research in Pairs, No 41235. The research of the first author is further supported by the Austrian Science Fund (FWF): W1230-N13 and P24028-N18, the Canada Research Chairs Program, and the European Research Council (ERC): No 725773 ``GroIsRan''.}}
\date{\today}
\maketitle

\begin{abstract}
\noindent We consider groups defined by non-empty balanced presentations with the property that each relator is of the form $R(x,y)$, where $x$ and $y$ are distinct generators and $R(\blanks)$ is determined by some fixed cyclically reduced word $R(a,b)$ that involves both $a$ and $b$. To every such presentation we associate a directed graph whose vertices correspond to the generators and whose arcs correspond to the relators. Under the hypothesis that the girth of the underlying undirected graph is at least 4, we show that the resulting groups are non-trivial and cannot be finite of rank~3 or higher. Without the hypothesis on the girth it is well known that both the trivial group and finite groups of rank~3 can arise.\medskip
\noindent \textbf{MSC2010 classes:} 20F05 (primary), 20E99 (secondary)
\end{abstract}

\section{Introduction}\label{sec:Intro}

The groups considered in the present paper fit into the following general framework: they are defined by finite presentations with the property that each relator is of the form $R(x,y)$, where $x$ and $y$ are distinct generators and $R(\blanks)$ is determined by some fixed cyclically reduced word $R(a,b)$ (in the free group generated by $a$ and $b$) that involves both $a$ and $b$. Prominent examples of these groups are \emph{right-angled Artin groups} (also known as \emph{graph groups} or \emph{partially commutative groups}), which arise when $R(a,b)=a^{-1}b^{-1}ab$.

Each of the above groups can be expressed in terms of a finite \emph{digraph}~$\Lambda$ with vertex set $V(\Lambda)$ and (directed) arc set $A(\Lambda)$. We use the convention that the arc set is an irreflexive relation on the vertex set so neither admit multiple arcs nor loops. The vertices $v\in V(\Lambda)$ correspond to the generators~$x_v$ and the arcs~$(u,v)\in A(\Lambda)$ correspond to the relators $R(x_u,x_v)$ so that the group~$G_{\Lambda}(R)$ is defined by the presentation
\[ P_{\Lambda}(R)=\pres{x_v\ (v\in V(\Lambda))}{R(x_u,x_v)\ ((u,v)\in A(\Lambda))}. \]
Our class of groups can therefore be thought of as a class of generalized graph groups or \emph{digraph groups}. In the case where $\Lambda$ is a directed $n$-cycle, \ie $V(\Lambda)=\{1,2,\ldots,n\}$ and $A(\Lambda) = \{ (1,2),(2,3),\ldots, (n,1)\}$, the corresponding presentation
\[ P_{\Lambda}(R)=\pres{x_1,\ldots,x_n}{R(x_1,x_2),R(x_2,x_3),\ldots,R(x_n,x_1)} \]
is an example of a \emph{cyclic presentation} (see, \egcomma\cite[Chapter~III, \S9]{JohnsonBook}). Cyclic presentations are special cases of \emph{balanced presentations}, which are presentations with an equal number of generators and relators. Since presentations with more generators than relators necessarily define infinite groups, which can be seen by abelianizing the groups, balanced presentations represent a borderline situation where the corresponding groups can be finite or infinite. In our setting $P_{\Lambda}(R)$ is balanced precisely when $\Lambda$ has an equal number of vertices and arcs. This is the case that we shall focus on in the present paper.

The \emph{rank} of a group $G$, denoted by $\mathrm{rank}(G)$, is the cardinality of a smallest generating set for~$G$. It is not known if there exists a balanced presentation defining a finite group of rank~4 or larger (see, \egcomma\cite[Problem~1]{HavasNewmanOBrien}) though it is known that rank~3 can be attained (see, \egcomma\cite[Chapter~III, \S8]{JohnsonBook}).

It is a consequence of the Golod--Shafarevich theorem~\cite{GolodShafarevich} that whenever a balanced presentation defines a finite nilpotent group, the latter has rank at most~3. Further, by \cite[Theorem~9~(ii)]{JohnsonWamsleyWright}, whenever a balanced presentation defines a finite group, be it nilpotent or not, the abelianization of that group must also have rank at most~3. Details can be found in~\cite[\S3--4]{JohnsonRobertson} and in the references therein. Let us finally record that a group has rank~0 if and only if it is the trivial group; balanced presentations of the trivial group are sought in connection with the Andrews--Curtis conjecture~\cite{AndrewsCurtis}.

Both balanced presentations of finite groups of rank~3 and non-empty balanced presentations of the trivial group can be found within our class of presentations $P_{\Lambda}(R)$. For example, if $\Lambda$ is the directed 3-cycle and $R(a,b)=a^{-1}bab^{-q}$, we obtain Mennicke's groups
\[ M(q,q,q)=\pres{x_1,x_2,x_3}{x_1^{-1}x_2x_1=x_2^q,x_2^{-1}x_3x_2=x_3^q,x_3^{-1}x_1x_3=x_1^q},\]
which are studied in~\cite{Mennicke} and appear as $G(1,q;1,q;1,q)$ in~\cite{Allcock}. For all $q\geq 3$ these groups are finite of rank~3~\cite{Mennicke} and for $q=2$ they are trivial~\cite{Higman}. If we modify this example and take the directed 2-cycle instead of the directed 3-cycle, the resulting groups will have similar properties: for all $q\geq 3$ they are finite of rank~2 and for $q=2$ they are trivial. The essential part of the proof can be found in~\cite{CampbellRobertson}. On the other hand, if we take a directed $n$-cycle ($n\geq 4$) and an arbitrary non-zero integer~$q$, the resulting groups will always be infinite. This follows, for example, from a general curvature argument due to Pride, see Corollary~\ref{cor:trianglefreeimpliesinfinite}. Note that for the directed 4-cycle and $q=2$ the resulting group is Higman's group~\cite{Higman}. A different example arises when~$\Lambda$ is the directed 3-cycle and $R(a,b)=b^{-1}ab (b^{q-2}a^{-1}b^{q+2})^{-1}$, where we obtain the groups
\[ J(q,q,q)=\hugepres{x_1,x_2,x_3}{\begin{array}{l}x_2^{-1}x_1x_2=x_2^{q-2}x_1^{-1}x_2^{q+2},\\ x_3^{-1}x_2x_3=x_3^{q-2}x_2^{-1}x_3^{q+2},\\ x_1^{-1}x_3x_1=x_1^{q-2}x_3^{-1}x_1^{q+2}\end{array}}\]
considered in~\cite{Johnson79} and \cite[page~70]{JohnsonBook}, which for all even $q\geq 2$ are finite of rank~3.

Pride showed in~\cite{Pride} that if $\Lambda$ is a directed $n$-cycle ($n\geq 4$) and $R(a,b)$ is a cyclically reduced word that involves both $a$ and $b$, then the resulting group $G_{\Lambda}(R)$ can never be finite of rank~3 or trivial. The precise statement is given in Theorem~\ref{thm:pridecycpres}, below. The following notational convention, partially introduced by Pride, will be used throughout this paper.

\begin{convention*}
Given a cyclically reduced word $R(a,b)$ that involves both $a$ and $b$, we use $\alpha$ and $-\beta$ to denote the exponent sums of $a$ and $b$ in $R(a,b)$, respectively, and $K$ to denote the group defined by the presentation $\pres{a,b}{R(a,b)}$. Up to cyclic permutation the word $R(a,b)$ is of the form $a^{\alpha_1}b^{\beta_1}\cdots a^{\alpha_t}b^{\beta_t}$ with $t\geq 1$ and $\alpha_i,\beta_i\in\Z\smallsetminus\{0\}$ \brkt{$1\leq i \leq t$}. We use $\delta_a$ and $\delta_b$ to denote the greatest common divisors $(\alpha_1,\ldots,\alpha_t)$ and $(\beta_1,\ldots,\beta_t)$, respectively.
\end{convention*}

\begin{theorem}[{\cite[Theorem~3]{Pride}}]\label{thm:pridecycpres}
Let $\Lambda$ be a directed $n$-cycle \brkt{$n\geq 4$} and let $R(a,b)$ be a cyclically reduced word that involves both $a$ and $b$. Then $G_{\Lambda}(R)$ is finite if and only if $\alpha\neq 0$, $\beta\neq 0$, $(\alpha,\beta)=1$, $\alpha^n-\beta^n\neq 0$, $a^\alpha=b^\beta$ in $K$, in which case $G_{\Lambda}(R)\cong\Z_{|\alpha^n-\beta^n|}$. In particular, if $G_{\Lambda}(R)$ is finite, then $\mathrm{rank}(G_{\Lambda}(R))=1$.
\end{theorem}

The purpose of the present paper is to generalize Theorem~\ref{thm:pridecycpres} from cyclic presentations to balanced presentations, \ie to the case where the digraph $\Lambda$ has an equal number of vertices and arcs. As shown by the examples, above, if the \emph{underlying undirected graph} of $\Lambda$, \ie the undirected graph obtained from $\Lambda$ by replacing each directed edge by an undirected edge, contains a cycle of length 2 or 3, the conclusion that $G_{\Lambda}(R)$ is neither finite of rank~3 nor trivial cannot always be obtained. We therefore impose the hypothesis that the girth of the underlying undirected graph of $\Lambda$ is at least 4. Corollary~\ref{maincor:rank<3} then provides the aspired generalization of Theorem~\ref{thm:pridecycpres}. This is a corollary to our main theorem, Theorem~\ref{mainthm:balanced}, which gives tight conditions that must be satisfied if $G_{\Lambda}(R)$ is a finite group.

Before stating it we introduce some terminology for digraphs. By a \emph{weakly connected} digraph we mean a digraph whose underlying undirected graph is connected. Vertices with positive outdegree and indegree zero are called \emph{sources}, vertices with positive indegree and outdegree zero are called \emph{sinks}, and vertices whose indegree and outdegree sum to one are called \emph{leaves}. In particular, every leaf must be either a source or a sink. In some cases we will \emph{recursively prune} all source leaves. This means that we remove the source leaves from the vertex set $V(\Lambda)$ and the arcs that are incident with these source leaves from the arc set $A(\Lambda)$. Afterwards, we consider the resulting digraph and repeat the previous step until we eventually arrive at a digraph $\Lambda_{s}$ without any source leaves. Because the initial digraph~$\Lambda$ is finite, this procedure is guaranteed to end. In the same way, recursively pruning all sink leaves yields a digraph~$\Lambda_{t}$ without any sink leaves. To state our results we will need to refer to certain classes of digraphs. These are defined in Figure~\ref{fig:digraphs}, below.
%
% ------------------------------
% Adjusting the page layout
% ------------------------------
%
\vfill

\begin{figure}[!ht]
\begin{center}
\includegraphics{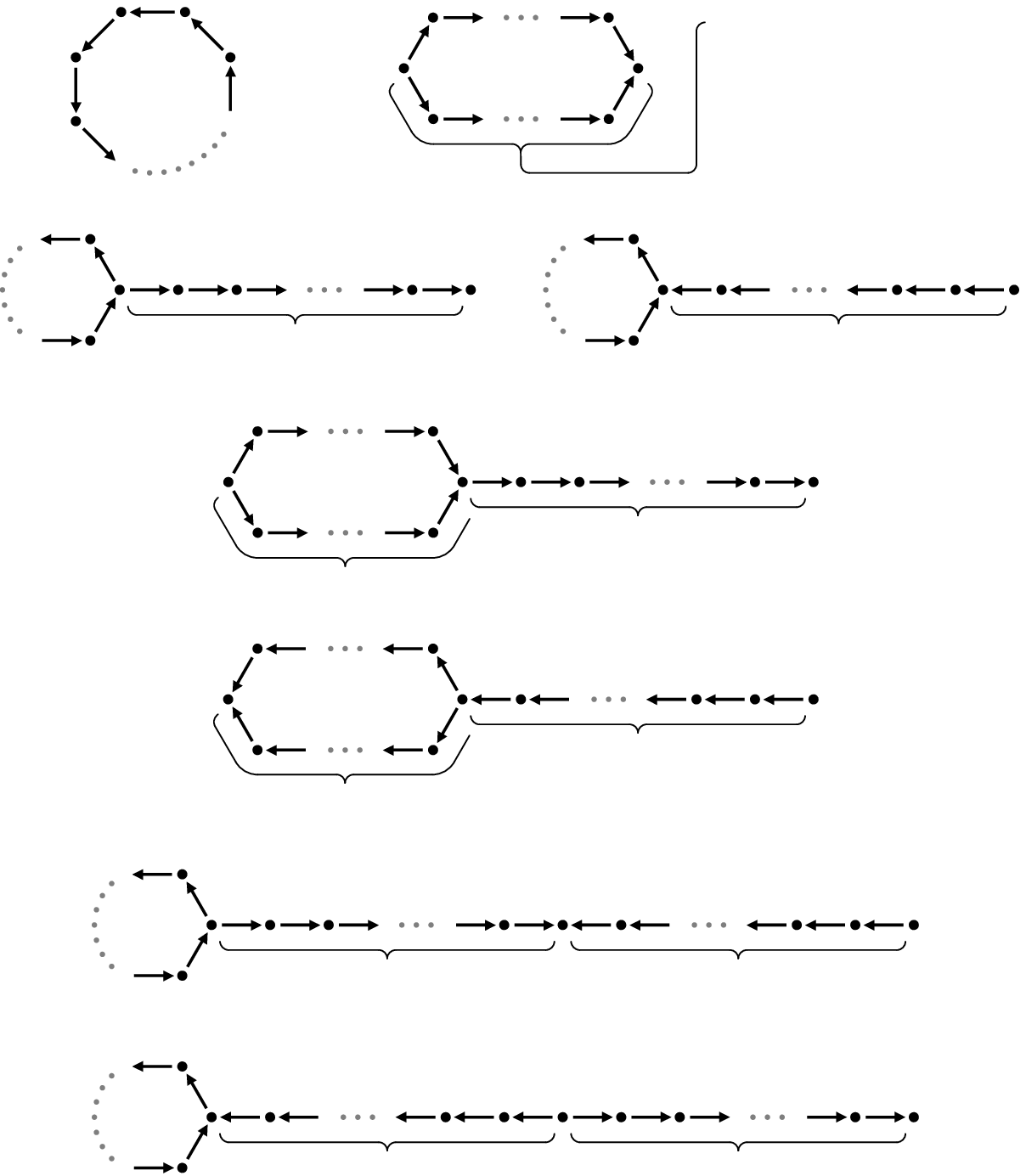}
\begin{picture}(0,0)

\put(-356.5,390){$\Lambda(n)$}
\put(-301.5,373.75){\small $n$}
\put(-307.5,362.25){\small arcs}

\put(-253.5,390){$\Lambda(n,d)$}
\put(-187,375){\small $n$ arcs}
\put(-105.5,373){\begin{minipage}{102pt} \baselineskip=12pt The shortest directed path from the source to the sink has $d$ arcs, so \end{minipage}}
\put(-105.5,345){\begin{minipage}{102pt} \[ \tfrac{n}{2}\geq d\geq 1. \] \end{minipage}}

\put(-235,312){$\Lambda(n;\rarw{m})$}
\put(-275.5,280){\small $m\geq 1$ arcs}
\put(-333.5,304.5){\small $n$}
\put(-339.5,293.5){\small arcs}

\put(-50,312){$\Lambda(n;\larw{m})$}
\put(-90.5,280){\small $m\geq 1$ arcs}
\put(-148.5,304.5){\small $n$}
\put(-154.5,293.5){\small arcs}

\put(-100,98.5){$\Lambda(n;\rarw{m},\larw{\ell})$}
\put(-302,88){\small $n$}
\put(-308,77){\small arcs}
\put(-244.25,64){\small $m\geq 1$ arcs}
\put(-120,64){\small $\ell\geq 1$ arcs}

\put(-100,33){$\Lambda(n;\larw{m},\rarw{\ell})$}
\put(-302,22.5){\small $n$}
\put(-308,11.5){\small arcs}
\put(-244.25,-1.5){\small $m\geq 1$ arcs}
\put(-120,-1.5){\small $\ell\geq 1$ arcs}

\put(-130,247.5){$\Lambda(n,d;\rarw{m})$}
\put(-246.5,233.5){\small $n$ arcs}
\put(-256,197.5){\small $\frac{n}{2}\geq d\geq 1$}
\put(-159,214.5){\small $m\geq 1$ arcs}

\put(-130,173){$\Lambda(n,d;\larw{m})$}
\put(-246.5,159){\small $n$ arcs}
\put(-256,123){\small $\frac{n}{2}\geq d\geq 1$}
\put(-159,140){\small $m\geq 1$ arcs}

\end{picture}
%
% ------------------------------
% Adjusting the page layout
% ------------------------------
%
\medskip\\
\caption{Classes of digraphs referred to in the statement of Theorem~\ref{mainthm:balanced}.}
\label{fig:digraphs}
\end{center}
\end{figure}

\begin{maintheorem}\label{mainthm:balanced}
Let $\Lambda$ be a non-empty finite digraph with an equal number of vertices and arcs whose underlying undirected graph has girth $n$ \brkt{$n\geq 4$} and let $R(a,b)$ be a cyclically reduced word that involves both $a$ and $b$ with exponent sums $\alpha$ and $-\beta$ in $a$ and $b$, respectively. If $G_{\Lambda}(R)$ is finite, then $\alpha\neq 0$, $\beta\neq 0$, $(\alpha,\beta)=1$, $\alpha^n-\beta^n\neq 0$, $a^\alpha=b^\beta$ in $K=\pres{a,b}{R(a,b)}$, $G_{\Lambda}(R)$ is non-trivial, and one of the following holds:
\begin{center}
\begin{tabular}{r@{~~}r@{~~}p{7.6cm}@{~}p{6.55cm}}
\brkt{1} & \multicolumn{2}{@{}l}{$|\alpha|\geq 2$, $|\beta|\geq 2$, and} & in which case \medskip \\
& \brkt{a} & $\Lambda=\Lambda(n)$ \mydots & $G_{\Lambda}(R)\cong\Z_{|\alpha^n-\beta^n|}$, \smallskip \\
& \brkt{b} & $\Lambda=\Lambda(n;\rarw{m})$ \brkt{$m\geq 1$} \mydots & $G_{\Lambda}(R)\cong\Z_{|\beta^m(\alpha^n-\beta^n)|}$, \smallskip \\
& \brkt{c} & $\Lambda=\Lambda(n;\larw{m})$ \brkt{$m\geq 1$} \mydots & $G_{\Lambda}(R)\cong\Z_{|\alpha^m(\alpha^n-\beta^n)|}$, \medskip \\
& \brkt{d} & $\Lambda=\Lambda(n,1)$, $\delta_a=\delta_b=1$ \mydots & $G_{\Lambda}(R)\cong K/\ngpres{a^{\alpha(\alpha^{n-2}-\beta^{n-2})}}^K$, \smallskip \\
&&& $G_{\Lambda}(R)^\mathrm{ab}\cong \Z_{|\alpha\beta(\alpha^{n-2}-\beta^{n-2})|}$, \medskip \\
& \brkt{e} & $\Lambda=\Lambda(n;\rarw{m},\larw{1})$ \brkt{$m\geq 1$}, $\delta_a=\delta_b=1$ \mydots & $G_{\Lambda}(R)\cong K/\ngpres{ b^{\beta^{m}(\alpha^n-\beta^n)} }^K$, \smallskip \\
&&& $G_{\Lambda}(R)^\mathrm{ab}\cong \Z_{|\alpha\beta^m(\alpha^{n}-\beta^{n})|}$, \medskip \\
& \brkt{f} & $\Lambda=\Lambda(n;\larw{m},\rarw{1})$ \brkt{$m\geq 1$}, $\delta_a=\delta_b=1$ \mydots & $G_{\Lambda}(R)\cong K/\ngpres{ a^{\alpha^{m}(\alpha^n-\beta^n)} }^K$, \smallskip \\
&&& $G_{\Lambda}(R)^\mathrm{ab}\cong \Z_{|\alpha^m\beta(\alpha^{n}-\beta^{n})|}$. \medskip
\end{tabular}
\begin{tabular}{r@{~~}r@{~~}p{7.6cm}@{~}p{6.55cm}}
\brkt{2} & \multicolumn{2}{@{}l}{$|\alpha|=1$, $|\beta|\geq 2$, and after recursively pruning} & \\
& \multicolumn{2}{@{}l}{all source leaves the digraph $\Lambda$ becomes} & in which case \medskip \\
&\brkt{a} & $\Lambda_{s}=\Lambda(n)$ \mydots & $G_{\Lambda}(R)\cong\Z_{|\alpha^n-\beta^n|}$, \smallskip \\
&\brkt{b} & $\Lambda_{s}=\Lambda(n;\rarw{m})$ \brkt{$m\geq 1$} \mydots & $G_{\Lambda}(R)\cong\Z_{|\beta^m(\alpha^n-\beta^n)|}$, \smallskip \\
&\brkt{c} & $\Lambda_{s}=\Lambda(n,d)$ \brkt{$\frac{n}{2}>d\geq 1$} \mydots & $G_{\Lambda}(R)\cong\Z_{|(\alpha\beta)^{n-d}-(\alpha\beta)^{d}|}$, \smallskip \\
&\brkt{d} & $\Lambda_{s}=\Lambda(n,d;\rarw{m})$ \brkt{$\frac{n}{2}>d\geq 1$, $m\geq 1$} \mydots & $G_{\Lambda}(R)\cong\Z_{|\beta^{m}((\alpha\beta)^{n-d}-(\alpha\beta)^{d})|}$, \smallskip \\
&\brkt{e} & $\Lambda_{s}=\Lambda(n;\larw{m},\rarw{\ell})$ \brkt{$m\geq 1$, $\ell\geq 1$} \mydots & $G_{\Lambda}(R)\cong\Z_{|\beta^{\ell}(\alpha^{n}-\beta^{n})|}$. \medskip
\end{tabular}
\begin{tabular}{r@{~~}r@{~~}p{7.6cm}@{~}p{6.55cm}}
\brkt{3} & \multicolumn{2}{@{}l}{$|\alpha|\geq 2$, $|\beta|=1$, and after recursively pruning} & \\
& \multicolumn{2}{@{}l}{all sink leaves the digraph $\Lambda$ becomes} & in which case \medskip \\
&\brkt{a} & $\Lambda_{t}=\Lambda(n)$ \mydots & $G_{\Lambda}(R)\cong\Z_{|\alpha^n-\beta^n|}$, \smallskip \\
&\brkt{b} & $\Lambda_{t}=\Lambda(n;\larw{m})$ \brkt{$m\geq 1$} \mydots & $G_{\Lambda}(R)\cong\Z_{|\alpha^m(\alpha^n-\beta^n)|}$, \smallskip \\
&\brkt{c} & $\Lambda_{t}=\Lambda(n,d)$ \brkt{$\frac{n}{2}>d\geq 1$} \mydots & $G_{\Lambda}(R)\cong\Z_{|(\alpha\beta)^{n-d}-(\alpha\beta)^{d}|}$, \smallskip \\
&\brkt{d} & $\Lambda_{t}=\Lambda(n,d;\larw{m})$ \brkt{$\frac{n}{2}>d\geq 1$, $m\geq 1$} \mydots & $G_{\Lambda}(R)\cong\Z_{|\alpha^{m}((\alpha\beta)^{n-d}-(\alpha\beta)^{d})|}$, \smallskip \\
&\brkt{e} & $\Lambda_{t}=\Lambda(n;\rarw{m},\larw{\ell})$ \brkt{$m\geq 1$, $\ell\geq 1$} \mydots & $G_{\Lambda}(R)\cong\Z_{|\alpha^{\ell}(\alpha^{n}-\beta^{n})|}$. \medskip \\
\brkt{4} & \multicolumn{2}{@{}l}{$|\alpha|=|\beta|=1$, in which case \mydots} & $G_{\Lambda}(R)\cong\Z_{2}$. \smallskip
\end{tabular}
\end{center}
\end{maintheorem}

\begin{remark}
Using our terminology, Theorem~\ref{thm:pridecycpres} is a statement about the digraphs $\Lambda(n)$ \brkt{$n\geq 4$} and thus follows from Cases \brkt{1a}, \brkt{2a}, \brkt{3a}, \brkt{4} of Theorem~\ref{mainthm:balanced}. Note that in Case \brkt{4} the numerical restrictions imply that $|\alpha^{n}-\beta^{n}|=2$, whence we also obtain $G_{\Lambda}(R)\cong\Z_{|\alpha^n-\beta^n|}$.
\end{remark}

Theorem~\ref{mainthm:balanced} has the following immediate corollary.

\begin{maincorollary}\label{maincor:rank<3}
With the notation of Theorem~\ref{mainthm:balanced}, if $G_{\Lambda}(R)$ is finite, then $\mathrm{rank}(G_{\Lambda}(R))\in\sset{1,2}$.
\end{maincorollary}

There exist infinitely many examples of relators $R(a,b)$ that satisfy the conditions of Cases (1d)--(1f) of Theorem~\ref{mainthm:balanced}, for example $R(a,b)=(ab)^qb$ ($q\geq 2$). However, all examples that we are aware of yield abelian groups $K$, and hence $G_\Lambda(R)$ is a finite cyclic group. For this reason we have been unable to construct any example where we cannot determine finiteness, or otherwise, of $G_\Lambda(R)$. We therefore pose:

\begin{question} \label{q:question}
Does there exist a word $R(a,b)=a^{\alpha_1}b^{\beta_1}\cdots a^{\alpha_t}b^{\beta_t}$ with $t\geq 1$, $\alpha_i,\beta_i\in\Z\smallsetminus\{0\}$ \brkt{$1\leq i \leq t$}, $\alpha=\sum_{i=1}^t\alpha_i$, $\beta=-\sum_{i=1}^t\beta_i$, $|\alpha|\geq 2$, $|\beta|\geq 2$, $(\alpha,\beta)=1$, $\delta_{a}=(\alpha_1,\ldots,\alpha_t)=1$, $\delta_{b}=(\beta_1,\ldots,\beta_t)=1$ such that $a^\alpha=b^\beta$ in $K=\pres{a,b}{R(a,b)}$ and $K$ is not abelian?
\end{question}

If Question~\ref{q:question} has a negative answer, then Theorem~\ref{mainthm:balanced} provides the following nice dichotomy directly generalizing Theorem~\ref{thm:pridecycpres}: \emph{If $\Lambda$ is a non-empty finite digraph with an equal number of vertices and arcs whose underlying undirected graph has girth at least 4 and $R(a,b)$ is a cyclically reduced word that involves both $a$ and $b$, then $G_{\Lambda}(R)$ is non-trivial and it is either finite cyclic or infinite.}

\begin{acknowledgement*}
We would like to thank the anonymous referee, whose remarks gave rise to several improvements of this paper. In particular, of the proofs presented in Section~\ref{subsec:proofcase23}.
\end{acknowledgement*}

\section{Pride's Property \texorpdfstring{$W_{1}$}{W1}}
\label{sec:pride}

In~\cite[page~246]{Pride} Pride introduced the following property: a two-generator group with generators $a$ and $b$ is said to have \emph{Property~$W_1$ \brkt{with respect to $a$ and $b$}} if no non-empty word of the form $a^kb^{-\ell}$ ($k,\ell\in\Z$) is equal to the identity in that group. In fact, not just Property~$W_1$ but \emph{Properties~$W_p$} ($p\in\mathbb{Z}$, $p\geq 1$) are defined there. As soon as the elements~$a$ and~$b$ have infinite order, Property~$W_{p}$ corresponds to the Gersten--Stallings angle $\sphericalangle(\gpres{a},\gpres{b};\sset{1})$ being at most~$\frac{\pi}{p+1}$ \cite{Stallings}. Under the hypothesis that the girth of the underlying undirected graph of $\Lambda$ is at least~4, the condition that $K$ has Property~$W_1$ can therefore be thought of as a condition of non-positive curvature. That non-positive curvature, in this sense, is a property that corresponds to infinite groups is a consequence of the following immediate corollary to \cite[Theorem~4]{Pride}. It forms a crucial ingredient to our methods.

\begin{corollary}[to {\cite[Theorem~4]{Pride}}]\label{cor:trianglefreeimpliesinfinite}
Let $\Lambda$ and $R(a,b)$ be as in Theorem~\ref{mainthm:balanced}. If $K$ has Property~$W_1$, then $G_{\Lambda}(R)$ is infinite.
\end{corollary}

It is therefore important to study groups that do not have Property~$W_1$. When we want to determine whether $K$ is of that kind, we first observe that if a non-empty word $a^kb^{-\ell}$ is equal to the identity in~$K$, then both $k\neq 0$ and $\ell\neq 0$. Indeed, if $k=0$, we have $b^\ell=1$ in $K$. Recall from the statement of Theorem~\ref{mainthm:balanced} that $R(a,b)$ is cyclically reduced and involves both~$a$ and~$b$. Therefore, by the Freiheitssatz for one-relator groups \cite{Magnus}, the element $b$ has infinite order in $K$, which implies that~$\ell=0$. Hence, $a^kb^{-\ell}$ is the empty word and we obtain a contradiction. Similarly, $\ell\neq 0$.

\begin{prop}[{\cite[page~248]{Pride}}]\label{prop:center}
If there exist $k,\ell\in\Z\smallsetminus\{0\}$ with $a^k=b^\ell$ in $K$, then $\alpha\neq 0$, $\beta\neq 0$, and $a^\alpha=b^\beta$ in $K$.
\end{prop}

Therefore, $K$ does not have Property~$W_1$ if and only if $\alpha\neq 0$, $\beta\neq 0$, and $a^\alpha=b^\beta$ in $K$. Note that the condition $a^\alpha=b^\beta$ in $K$ is decidable since the word problem is solvable for one-relator groups.\medskip

\begin{excursion}
\label{exc:magnus}
Property~$W_1$ is further related to one-relator groups whose Magnus subgroups have exceptional intersection; such groups are studied in~\cite{Collins04}, \cite{Collins08}, \cite{Howie05}, \cite{EdjvetHowie08}. More precisely, the Magnus subgroups $\gpres{a}$ and $\gpres{b}$ of $K$ are said to have \emph{exceptional intersection} if $\gpres{a}\cap\gpres{b}\cong\Z$ (see, \egcomma\cite{Collins04}). The following characterization of two-generator one-relator groups whose Magnus subgroups have exceptional intersection, and the connection with Property~$W_1$, does not seem to have been recorded explicitly before.

\begin{lemma}\label{lem:notW1isMI}
The Magnus subgroups $\gpres{a}$ and $\gpres{b}$ of $K$ have exceptional intersection if and only if $\alpha\neq 0$, $\beta\neq 0$, and $a^\alpha=b^\beta$ in $K$.
\end{lemma}

\begin{proof}
If $\alpha\neq 0$, $\beta\neq 0$, and $a^\alpha=b^\beta$ in $K$, then $\Z\cong\gpres{a^\alpha}=\gpres{b^\beta}\leq\gpres{a}\cap\gpres{b}\leq\gpres{a}\cong\Z$, so $\gpres{a}\cap\gpres{b}\cong\Z$. Conversely, if $\gpres{a}\cap\gpres{b}\cong\Z$, then an arbitrary non-trivial element of this intersection is of the form $a^k$ for some $k\in\Z\smallsetminus\{0\}$ and simultaneously of the form $b^{\ell}$ for some $\ell\in\Z\smallsetminus\{0\}$. So $a^k=b^\ell$ in $K$ and, by Proposition~\ref{prop:center}, we have $\alpha\neq 0$, $\beta\neq 0$, and $a^\alpha=b^\beta$ in $K$.
\end{proof}
\end{excursion}
%
% ------------------------------
% Adjusting the page layout
% ------------------------------
%
\enlargethispage{6pt}

Question~\ref{q:question} therefore concerns groups $K$ where $\gpres{a}$ and $\gpres{b}$ have exceptional intersection.

\section{Proving Theorem~\ref{mainthm:balanced}}
\label{sec:proof}

\subsection{Strategy}

Before giving the actual proof of Theorem~\ref{mainthm:balanced} we sketch the strategy and discuss some lemmas. Lemma~\ref{lem:trackingargument}, below, is the first of these and will serve as a general tool to simplify the presentations that arise in the subsequent parts of the paper. (Recall that we can always suppose that $\alpha\neq 0$, $\beta\neq 0$, and $a^{\alpha}=b^{\beta}$ in $K$ for otherwise, by Corollary~\ref{cor:trianglefreeimpliesinfinite} and Proposition~\ref{prop:center}, the group~$K$ has Property~$W_{1}$ and thus $G_{\Lambda}(R)$ is infinite.) In Section~\ref{subsec:proofcase1} we turn to the setting of Theorem~\ref{mainthm:balanced} and suppose that $|\alpha|\geq 2$ and $|\beta|\geq 2$. The strategy is to obtain conditions that must be satisfied if $G_{\Lambda}(R)$ is a finite group. We then analyse $G_{\Lambda}(R)$ under these conditions. In some cases we are able to show that $G_{\Lambda}(R)$ is a finite cyclic group, in others at least that $\mathrm{rank}(G_{\Lambda}(R))\in\sset{1,2}$. Then, in Section~\ref{subsec:proofcase23}, we consider the remaining cases and suppose without loss of generality that $|\alpha|=1$ and $|\beta|\geq 1$. The fact that $|\alpha|=1$ enables us to recursively prune the digraphs without changing the isomorphism types of the resulting groups.

\begin{lemma}\label{lem:trackingargument}
Let $R(a,b)$ be a word such that $a^{\alpha}=b^{\beta}$ in $K$ and let $G$ be a group defined by a presentation $\pres{\mathcal{X}}{\mathcal{R}}$. Further suppose that there are distinct generators $x_{i},x_{j}\in\mathcal{X}$ such that $R(x_{i},x_{j})\in\mathcal{R}$. Then the following hold:
\begin{itemize}
\item[\brkt{a}] If $x_{i}^{\gamma}\in\mathcal{R}$ for some $\gamma\in\mathbb{Z}$ with $(\alpha,\gamma)=1$, then every $p\in\mathbb{Z}$ with $p\alpha\equiv 1~(\text{mod}~\gamma)$ yields a new presentation $\pres{\mathcal{X}\smallsetminus\{x_{i}\}}{\mathcal{S}}$ of $G$. The relators $\mathcal{S}$ are obtained from $\mathcal{R}$ by removing $R(x_{i},x_{j})$ and~$x_{i}^{\gamma}$, replacing all remaining occurrences of $x_{i}$ by $x_{j}^{p\beta}$, and adjoining~$x_{j}^{\beta\gamma}$.
\item[\brkt{b}] If $x_{j}^{\gamma}\in\mathcal{R}$ for some $\gamma\in\mathbb{Z}$ with~$(\beta,\gamma)=1$, then every $p\in\mathbb{Z}$ with $p\beta\equiv 1~(\text{mod}~\gamma)$ yields a new presentation $\pres{\mathcal{X}\smallsetminus\{x_{j}\}}{\mathcal{S}}$ of $G$. The relators $\mathcal{S}$ are obtained from $\mathcal{R}$ by removing $R(x_{i},x_{j})$ and~$x_{j}^{\gamma}$, replacing all remaining occurrences of $x_{j}$ by $x_{i}^{p\alpha}$, and adjoining~$x_{i}^{\alpha\gamma}$.
\end{itemize}
\end{lemma}

\begin{proof}
The proofs of the two parts are similar and one only needs to interchange the roles of $x_i$ and~$x_j$ and the roles of $\alpha$ and $\beta$ (wherever necessary). We therefore prove Part (a) only. For simplicity of notation let $\mathcal{R}_{1}$ be obtained from $\mathcal{R}$ by removing $R(x_{i},x_{j})$ and $x_{i}^{\gamma}$. Further, let $\mathcal{R}_{2}$ be obtained from $\mathcal{R}_{1}$ by replacing all occurrences of $x_{i}$ by $x_{j}^{p\beta}$ so that $\mathcal{S}$ is obtained from $\mathcal{R}_{2}$ by adjoining~$x_{j}^{\beta\gamma}$. Note that, since $a^{\alpha}=b^{\beta}$ in $K$, we have
\[ G=\pres{\mathcal{X}}{\mathcal{R}}=\pres{\mathcal{X}}{\mathcal{R}_{1},R(x_{i},x_{j}),x_{i}^{\gamma}}=\pres{\mathcal{X}}{\mathcal{R}_{1},R(x_{i},x_{j}),x_{i}^{\alpha}=x_{j}^{\beta},x_{i}^{\gamma}}. \]
Because $p\alpha\equiv 1~(\text{mod}~\gamma)$, there is an integer $q\in\mathbb{Z}$ such that $p\alpha+q\gamma=1$. Moreover, $p\alpha\equiv 1~(\text{mod}~\gamma)$ implies that $x_{i}=x_{i}^{p\alpha}=x_{j}^{p\beta}$ in $G$. This allows us to adjoin the relation $x_{i}=x_{j}^{p\beta}$ and to eliminate the generator $x_{i}$ as follows:
\begin{alignat*}{1}
G
&=\pres{\mathcal{X}}{\mathcal{R}_{1},R(x_{i},x_{j}),x_{i}^{\alpha}=x_{j}^{\beta},x_{i}^{\gamma},x_{i}=x_{j}^{p\beta}}\\
&=\pres{\mathcal{X}}{\mathcal{R}_{2},R(x_{j}^{p\beta},x_{j}),x_{j}^{p\alpha\beta}=x_{j}^{\beta},x_{j}^{p\beta\gamma},x_{i}=x_{j}^{p\beta}}\\
&=\pres{\mathcal{X}\smallsetminus\{x_{i}\}}{\mathcal{R}_{2},R(x_{j}^{p\beta},x_{j}),x_{j}^{p\alpha\beta}=x_{j}^{\beta},x_{j}^{p\beta\gamma}}\\
&=\pres{\mathcal{X}\smallsetminus\{x_{i}\}}{\mathcal{R}_{2},x_{j}^{p\alpha\beta}=x_{j}^{\beta},x_{j}^{p\beta\gamma}}\\
&=\pres{\mathcal{X}\smallsetminus\{x_{i}\}}{\mathcal{R}_{2},x_{j}^{(1-p\alpha)\beta},x_{j}^{p\beta\gamma}}\\
&=\pres{\mathcal{X}\smallsetminus\{x_{i}\}}{\mathcal{R}_{2},x_{j}^{q\beta\gamma},x_{j}^{p\beta\gamma}}.
\end{alignat*}
The last two relators can be subsumed to a single one of the form $x_{j}^{r}$, where $r$ is the greatest common divisor of $p\beta\gamma$ and $q\beta\gamma$. Recall that $p\alpha+q\gamma=1$, so $(p,q)=1$ and $r=(p\beta\gamma,q\beta\gamma)=|(p,q)\beta\gamma|=|\beta\gamma|$. Therefore,
\[ G=\pres{\mathcal{X}\smallsetminus\{x_{i}\}}{\mathcal{R}_{2},x_{j}^{\beta\gamma}}=\pres{\mathcal{X}\smallsetminus\{x_{i}\}}{\mathcal{S}}. \]
\end{proof}

Let us address one application of this lemma to presentations of the form $P_{\Lambda}(R)$. If there are relators $R(x_{i_{1}},x_{i_{2}}),R(x_{i_{2}},x_{i_{3}}),\ldots,R(x_{i_{n-1}},x_{i_{n}})\in\mathcal{R}$, say arising from a directed path in $\Lambda$, and $x_{i_{1}}^{\gamma}\in\mathcal{R}$ for some $\gamma\in\mathbb{Z}$ with $(\alpha,\gamma)=1$, then we can apply Lemma~\ref{lem:trackingargument}\,(a) to remove the generator $x_{i_{1}}$, under suitable modifications of the remaining relators. Since $(\alpha,\beta)=1$ and $(\alpha,\gamma)=1$, we have $(\alpha,\beta\gamma)=1$. So the presence of the adjoined relator $x_{i_{2}}^{\beta\gamma}$ allows us to apply Lemma~\ref{lem:trackingargument}\,(a) again to remove the generator $x_{i_{2}}$. Inductively, we can remove all the generators $x_{i_{1}},\ldots,x_{i_{n-1}}$.

\begin{remark}\label{rem:reflection}
We will occasionally make use of a \emph{reflection principle}: if $\Lambda$ is any digraph and $R(a,b)$ is any word, then we may consider the digraph $\Lambda'$ that is obtained from $\Lambda$ by reversing the direction of each arc and the word $R\hspace{1pt}'(a,b)$ that is obtained from $R(a,b)$ by interchanging $a$ and $b$ and further replacing every letter by its inverse so that also $\alpha$ and $\beta$ are interchanged (without any change of sign). Then, by definition, $G_{\Lambda}(R)\cong G_{\Lambda'}(R\hspace{1pt}')$.
\end{remark}

\subsection{Cases where \texorpdfstring{$|\alpha|\geq 2$}{|alpha|>=2} and \texorpdfstring{$|\beta|\geq 2$}{|beta|>=2}}
\label{subsec:proofcase1}

The proof of the following lemma uses the notion of killing. When we are given a group $G=\pres{\mathcal{X}}{\mathcal{R}}$ and \emph{kill} a generator $x\in\mathcal{X}$, we simply adjoin the relator $x$. The generator $x$ can then be removed by a Tietze transformation. It is clear that $G$ maps onto the resulting group. We note that, unlike all later lemmas, Lemma~\ref{lem:M2allowedbalancedgraphs} does not require the hypothesis that $a^{\alpha}=b^{\beta}$ in $K$, and so it can be applied independently of Corollary~\ref{cor:trianglefreeimpliesinfinite} and Proposition~\ref{prop:center}.

\begin{lemma}\label{lem:M2allowedbalancedgraphs}
Let $\Lambda$ and $R(a,b)$ be as in Theorem~\ref{mainthm:balanced}. Further suppose that $\Lambda$ is weakly connected and that $|\alpha|\geq 2$ and $|\beta|\geq 2$. If $G_{\Lambda}(R)$ is finite, then $(\alpha,\beta)=1$ and $\Lambda$ is one of the following digraphs:
\begin{center}
\begin{tabular}{l@{\hspace{10mm}}l@{\hspace{10mm}}l}
\brkt{a} $\Lambda(n)$, & \brkt{b} $\Lambda(n;\rarw{m})$ \brkt{$m\geq 1$}, & \brkt{c} $\Lambda(n;\larw{m})$ \brkt{$m\geq 1$}, \\
\brkt{d} $\Lambda(n,1)$, & \brkt{e} $\Lambda(n;\rarw{m},\larw{1})$ \brkt{$m\geq 1$}, & \brkt{f} $\Lambda(n;\larw{m},\rarw{1})$ \brkt{$m\geq 1$}.
\end{tabular}
\end{center}
\end{lemma}

\begin{proof}
Suppose that $G_{\Lambda}(R)$ is finite. Because $\Lambda$ is a non-empty finite digraph with an equal number of vertices and arcs that is weakly connected, the underlying undirected graph of $\Lambda$ has precisely one cycle. By the hypothesis on the girth, this cycle must have length $n$ ($n\geq 4$), whence there exist two distinct vertices $u,w\in V(\Lambda)$ that are not connected by an arc. Consider the presentation $P_{\Lambda}(R)$. Killing all generators $x_v$ ($v \in V(\Lambda)\smallsetminus \{u,w\}$) yields a presentation $\pres{x_u,x_w}{\mathcal{R}}$ where
\[ \mathcal{R}\subseteq\sset{R(x_u,1),R(1,x_u),R(x_w,1),R(1,x_w)}=\sset{x_u^\alpha,x_u^\beta,x_w^\alpha,x_w^\beta}. \]
Further adjoining the relators $x_u^{(\alpha,\beta)}$ and $x_w^{(\alpha,\beta)}$ gives that $G_{\Lambda}(R)$ maps onto
\[ \pres{x_u,x_w}{\mathcal{R},x_u^{(\alpha,\beta)},x_w^{(\alpha,\beta)}}=\pres{x_u,x_w}{x_u^{(\alpha,\beta)},x_w^{(\alpha,\beta)}}\cong\Z_{(\alpha,\beta)}\ast\Z_{(\alpha,\beta)}, \]
which is infinite if $(\alpha,\beta)\neq 1$. Since $G_{\Lambda}(R)$ is finite, we have $(\alpha,\beta)=1$. This is the first statement of the lemma. For the second one we make two observations of similar flavour.

Suppose that there are a source $u\in V(\Lambda)$ and a sink $w\in V(\Lambda)$ that are not connected by an arc. Killing all generators $x_v$ ($v \in V(\Lambda)\smallsetminus \{u,w\}$) gives that $G_{\Lambda}(R)$ maps onto $\Z_{|\alpha|}\ast\Z_{|\beta|}$, which is infinite since $|\alpha|\geq 2$ and $|\beta|\geq 2$. Thus, we can assume that there is an arc between every source and every sink. Next, suppose that there are distinct vertices $u,w\in V(\Lambda)$ that are both sources (resp.\ both sinks). Clearly, these vertices cannot be connected by an arc. Killing all generators $x_v$ ($v \in V(\Lambda)\smallsetminus \{u,w\}$) gives that $G_{\Lambda}(R)$ maps onto $\Z_{|\alpha|}\ast\Z_{|\alpha|}$ (resp.\ $\Z_{|\beta|}\ast\Z_{|\beta|}$), which is infinite. Thus, we can assume that $\Lambda$ has at most one source and at most one sink. With these two restrictions in mind we can investigate the possible digraphs $\Lambda$. Let $\sigma$ and $\tau$ be the numbers of sources and sinks, respectively. Moreover, let $\sigma_1$ be the number of source leaves and let $\tau_1$ be the number of sink leaves. Then $0\leq \sigma_1\leq \sigma\leq 1$ and $0\leq \tau_1\leq \tau\leq 1$.

If $\sigma=\tau=0$, then $\sigma_1=\tau_1=0$ and $\Lambda$ is the directed $n$-cycle, \ie the digraph $\Lambda(n)$ of Figure~\ref{fig:digraphs}. If $\sigma=0$ and $\tau=1$, then $\sigma_1=0$ and $\tau_1\in\{0,1\}$. If $\tau_1=0$, then the underlying undirected graph of $\Lambda$ is a cycle, which is impossible because $\Lambda$ has fewer sources than sinks. Therefore, $\tau_1=1$ and $\Lambda=\Lambda(n;\rarw{m})$ ($m\geq 1$). In a similar way, if $\sigma=1$ and $\tau=0$, then we have $\Lambda=\Lambda(n;\larw{m})$ ($m\geq 1$).

Suppose then that $\sigma=\tau=1$ and so either $\sigma_1=\tau_1=0$ or $\sigma_1=1$, $\tau_1=0$ or $\sigma_1=0$, $\tau_1=1$. The case $\sigma_1=\tau_1=1$ cannot occur for otherwise the restriction that there is an arc between every source and every sink implies that $\Lambda$ is the digraph consisting of two vertices and one arc between them, and thus has more vertices than arcs. The restriction further yields that if $\sigma_1=\tau_1=0$, then $\Lambda=\Lambda(n,1)$. In a similar way, if $\sigma_1=0$, $\tau_1=1$ or $\sigma_1=1$, $\tau_1=0$, we have $\Lambda=\Lambda(n;\larw{m},\rarw{1})$ ($m\geq 1$) or $\Lambda=\Lambda(n;\rarw{m},\larw{1})$ ($m\geq 1$), respectively.
\end{proof}

\begin{excursion}
The statement of Lemma~\ref{lem:M2allowedbalancedgraphs} remains valid when we replace the hypothesis ``$n\geq 4$'' by ``$n\geq 3$ and the underlying undirected graph of $\Lambda$ is not a triangle''. However, for triangles the conclusion that $(\alpha,\beta)=1$ does not hold; the group $J(2,2,2)$ given in the introduction is finite and serves as a counterexample.

Despite this we cannot make the same replacement of hypotheses in Corollary~\ref{maincor:rank<3}. If we did, then the trivial group could arise. For example, let $\Lambda=\Lambda(3;\rarw{1})$ and $R(a,b)=a^{-1}bab^{-2}$. In $G_{\Lambda}(R)$ the three generators corresponding to the vertices of the directed $3$-cycle are trivial by \cite{Higman}, as mentioned in the introduction, which implies that the fourth generator must be trivial, too. It would be interesting to see whether all finite groups $G_{\Lambda}(R)$ arising when $n=3$ and the underlying undirected graph of $\Lambda$ is not a triangle satisfy $\mathrm{rank}(G_{\Lambda}(R))\in\sset{0,1,2}$.
\end{excursion}

We shall now consider the group $G_{\Lambda}(R)$, and in some cases also the abelianization $G_{\Lambda}(R)^{\mathrm{ab}}$, for each of the digraphs $\Lambda$ in Lemma~\ref{lem:M2allowedbalancedgraphs}\,(a)--(f) under its conclusion that $(\alpha,\beta)=1$ and the further hypothesis that $a^\alpha=b^\beta$ in $K$. The following lemma was stated without proof in~\cite[page~248]{Pride} and the omitted argument was given in~\cite[Lemma~3.4]{BogleyWilliams}. Nevertheless, we include a different proof here as a showcase for the methods used in this paper and for later reference.

\begin{lemma}[{\cite[page~248]{Pride}, \cite[Lemma~3.4]{BogleyWilliams}}]\label{lem:M1Lambda0000}
Let $R(a,b)$ be as in Theorem~\ref{mainthm:balanced}. Further suppose that $(\alpha,\beta)=1$ and $a^\alpha=b^\beta$ in $K$. If $\Lambda=\Lambda(n)$ \brkt{$n\geq 2$}, then $G_{\Lambda}(R)\cong\Z_{|\alpha^n-\beta^n|}$.
\end{lemma}

\begin{proof}
Let $V(\Lambda) = \{1,\ldots,n\}$ and $A(\Lambda) = \{ (1,2),(2,3),\ldots, (n,1)\}$, whence the group $G_{\Lambda}(R)$ is defined by the presentation
\[ \pres{x_{1},\ldots,x_{n}}{R(x_{1},x_{2}),R(x_{2},x_{3}),\ldots,R(x_{n},x_{1})}. \]
Because $a^\alpha=b^\beta$ in $K$, we know that whenever the relator $R(x_i,x_j)$ appears in the above list, the respective equation $x_{i}^\alpha=x_{j}^\beta$ holds in $G_{\Lambda}(R)$. Therefore,
\[ x_{1}^{\alpha^{n}}=x_{2}^{\alpha^{n-1}\beta}=x_{3}^{\alpha^{n-2}\beta^{2}}=\ldots=x_{n}^{\alpha\beta^{n-1}}=x_{1}^{\beta^{n}}. \]
We set $\gamma=\alpha^{n}-\beta^{n}$ and obtain $x_{1}^{\gamma}=1$ in $G_{\Lambda}(R)$. Adjoining the relator $x_{1}^{\gamma}$ yields
\[ G_{\Lambda}(R)=\pres{x_{1},\ldots,x_{n}}{x_{1}^{\gamma},R(x_{1},x_{2}),R(x_{2},x_{3}),\ldots,R(x_{n},x_{1})}. \]
Since $(\alpha,\beta)=1$, also $(\alpha,\gamma)=1$. We can thus apply Lemma~\ref{lem:trackingargument}\,(a) to simplify the presentation. Before doing so let us observe that also $(\alpha,\beta\gamma)=\ldots=(\alpha,\beta^{n-2}\gamma)=1$. By choosing an integer $p\in\mathbb{Z}$ such that $p\alpha\equiv 1~(\text{mod}~\beta^{n-2}\gamma)$, the congruence $p\alpha\equiv 1$ simultaneously holds modulo $\gamma$, $\beta\gamma$, \ldots, $\beta^{n-2}\gamma$. Now, an iterated application of Lemma~\ref{lem:trackingargument}\,(a) yields
\begin{alignat*}{1}
G_{\Lambda}(R)
&=\pres{x_{1},\ldots,x_{n}}{x_{1}^{\gamma},R(x_{1},x_{2}),R(x_{2},x_{3}),R(x_{3},x_{4}),\ldots,R(x_{n},x_{1})}\\
&=\pres{x_{2},\ldots,x_{n}}{x_{2}^{\beta\gamma},R(x_{2},x_{3}),R(x_{3},x_{4}),\ldots,R(x_{n},x_{2}^{p\beta})}\\
&=\pres{x_{3},\ldots,x_{n}}{x_{3}^{\beta^{2}\gamma},R(x_{3},x_{4}),\ldots,R(x_{n},x_{3}^{p^{2}\beta^{2}})}\\
&=\ldots=\pres{x_{n}}{x_{n}^{\beta^{n-1}\gamma},R(x_{n},x_{n}^{p^{n-1}\beta^{n-1}})}\\
&=\pres{x_{n}}{x_{n}^{\beta^{n-1}\gamma},x_{n}^{\alpha-\beta(p^{n-1}\beta^{n-1})}}.
\end{alignat*}
Therefore, $G_{\Lambda}(R)\cong\mathbb{Z}_{r}$, where
\[ r=(\beta^{n-1}\gamma,\alpha-\beta(p^{n-1}\beta^{n-1}))=(\gamma,\alpha-p^{n-1}\beta^{n}). \]
In order to evaluate the rightmost term, observe that
\[ p^{n-1}\beta^{n}\equiv (p\alpha)p^{n-1}\beta^{n}=\alpha p^{n}\beta^{n}=\alpha p^{n}(\alpha^{n}-\gamma)\equiv \alpha p^{n}\alpha^{n}=\alpha(p\alpha)^{n}\equiv \alpha 1^{n}=\alpha~(\text{mod}~\gamma). \]
So $\gamma$ divides $\alpha-p^{n-1}\beta^{n}$, which implies that $r=|\gamma|=|\alpha^n-\beta^n|$, as required.
\end{proof}

\begin{lemma}\label{lem:M1Lambda1010or0101}
Let $R(a,b)$ be as in Theorem~\ref{mainthm:balanced}. Further suppose that $(\alpha,\beta)=1$ and $a^\alpha=b^\beta$ in~$K$. Then the following hold:
\begin{itemize}
\item[\brkt{a}] If $\Lambda=\Lambda(n;\rarw{m})$ \brkt{$n\geq 2$, $m\geq 1$}, then $G_{\Lambda}(R)\cong\Z_{|\beta^m(\alpha^n-\beta^n)|}$.
\item[\brkt{b}] If $\Lambda=\Lambda(n;\larw{m})$ \brkt{$n\geq 2$, $m\geq 1$}, then $G_{\Lambda}(R)\cong\Z_{|\alpha^m(\alpha^n-\beta^n)|}$.
\end{itemize}
\end{lemma}

\begin{proof}
We prove Part~(a) only. Part~(b) can either be proved similarly or it can be deduced from Part\,(a) using the reflection principle addressed in Remark~\ref{rem:reflection}. The group $G_{\Lambda}(R)$ is defined by the presentation
\[ \hugepres{
\begin{array}{l}
x_1,\ldots,x_n,\\
y_1,\ldots, y_m
\end{array}
}{
\begin{array}{l}
R(x_1,x_2),R(x_2,x_3),\ldots,R(x_n,x_1), \\
R(x_n,y_1),R(y_1,y_2),\ldots,R(y_{m-1},y_m)
\end{array}
}. \]
We set $\gamma=\alpha^n-\beta^n$ and apply precisely the same transformations as in the proof of Lemma~\ref{lem:M1Lambda0000}. The fact that there are further generators and relators does not affect the validity of the transformations. What remains is
\[ G_{\Lambda}(R)=\pres{x_n,y_1,\ldots,y_m}{x_n^{\gamma},R(x_n,y_1),R(y_1,y_2),\ldots,R(y_{m-1},y_m)}. \]
We continue simplifying this presentation. Choose an integer $p\in\mathbb{Z}$ such that $p\alpha\equiv 1~(\text{mod}~\beta^{m-1}\gamma)$. Hence, the congruence $p\alpha\equiv 1$ simultaneously holds modulo $\gamma$, $\beta\gamma$, \ldots, $\beta^{m-1}\gamma$. Now, an iterated application of Lemma~\ref{lem:trackingargument}\,(a) yields
\begin{alignat*}{1}
G_{\Lambda}(R)
&=\pres{x_n,y_1,\ldots, y_m}{x_n^{\gamma},R(x_n,y_1),R(y_1,y_2),R(y_2,y_3),\ldots,R(y_{m-1},y_m)} \\
&=\pres{y_1,\ldots,y_m}{y_1^{\beta\gamma},R(y_1,y_2),R(y_2,y_3),\ldots,R(y_{m-1},y_m)} \\
&=\pres{y_2,\ldots,y_m}{y_2^{\beta^{2}\gamma},R(y_2,y_3),\ldots,R(y_{m-1},y_m)} \\
&=\ldots=\pres{y_{m-1}, y_m}{y_{m-1}^{\beta^{m-1}\gamma},R(y_{m-1},y_m)}=\pres{y_m}{y_{m}^{\beta^{m}\gamma}}.
\end{alignat*}
Therefore, $G_{\Lambda}(R)\cong\Z_{|\beta^m\gamma|}=\Z_{|\beta^m(\alpha^n-\beta^n)|}$.
\end{proof}

\begin{lemma}\label{lem:M1Lambda1100}
Let $R(a,b)$ be as in Theorem~\ref{mainthm:balanced}. Further suppose that $(\alpha,\beta)=1$ and $a^\alpha=b^\beta$ in~$K$. If $\Lambda=\Lambda(n,1)$ \brkt{$n\geq 3$}, then
\[ G_{\Lambda}(R)\cong K/\ngpres{a^{\alpha(\alpha^{n-2}-\beta^{n-2})}}^K\quad\text{and}\quad G_{\Lambda}(R)^{\mathrm{ab}}\cong\Z_{|\alpha\beta(\alpha^{n-2}-\beta^{n-2})|}. \]
If, in addition, $|\alpha|\neq 1$, $|\beta|\neq 1$, and $G_\Lambda(R)$ is finite, then $\delta_a=\delta_b=1$.
\end{lemma}

\begin{proof}
Let $V(\Lambda) = \{1,\ldots,n\}$. Without loss of generality let $n$ be the source and $n-1$ be the sink so that the group $G_{\Lambda}(R)$ is defined by the presentation
\[ \pres{x_{1},\ldots,x_{n}}{R(x_{1},x_{2}),R(x_{2},x_{3}),\ldots,R(x_{n-2},x_{n-1}),R(x_{n},x_{n-1}),R(x_{n},x_{1})}. \]
Because $a^\alpha=b^\beta$ in $K$, we know that whenever the relator $R(x_i,x_j)$ appears in the above list, the respective equation $x_{i}^\alpha=x_{j}^\beta$ holds in $G_{\Lambda}(R)$. Therefore,
\[ x_{1}^{\alpha^{n-2}}=x_{2}^{\alpha^{n-3}\beta}=x_{3}^{\alpha^{n-4}\beta^{2}}=\ldots=x_{n-2}^{\alpha\beta^{n-3}}=x_{n-1}^{\beta^{n-2}}=x_{n}^{\alpha\beta^{n-3}}=x_{1}^{\beta^{n-2}}. \]
We set $\gamma=\alpha^{n-2}-\beta^{n-2}$ and obtain $x_{1}^{\gamma}=1$ in $G_{\Lambda}(R)$. Adjoining this relator yields
\[ G_{\Lambda}(R)=\pres{x_{1},\ldots,x_{n}}{x_{1}^{\gamma},R(x_{1},x_{2}),R(x_{2},x_{3}),\ldots,R(x_{n-2},x_{n-1}),R(x_{n},x_{n-1}),R(x_{n},x_{1})}. \]
As in the previous proofs, choose an integer $p\in\mathbb{Z}$ such that $p\alpha\equiv 1~(\text{mod}~\beta^{n-3}\gamma)$. Hence, the congruence $p\alpha\equiv 1$ simultaneously holds modulo $\gamma$, $\beta\gamma$, \ldots, $\beta^{n-3}\gamma$. Now, an iterated application of Lemma~\ref{lem:trackingargument}\,(a) yields
\begin{alignat*}{1}
G_{\Lambda}(R)
&=\pres{x_{1},\ldots,x_{n}}{x_{1}^{\gamma},R(x_{1},x_{2}),R(x_{2},x_{3}),R(x_{3},x_{4}),\ldots,R(x_{n-2},x_{n-1}),R(x_{n},x_{n-1}),R(x_{n},x_{1})}\\
&=\pres{x_{2},\ldots,x_{n}}{x_{2}^{\beta\gamma},R(x_{2},x_{3}),R(x_{3},x_{4}),\ldots,R(x_{n-2},x_{n-1}),R(x_{n},x_{n-1}),R(x_{n},x_{2}^{p\beta})}\\
&=\pres{x_{3},\ldots,x_{n}}{x_{3}^{\beta^{2}\gamma},R(x_{3},x_{4}),\ldots,R(x_{n-2},x_{n-1}),R(x_{n},x_{n-1}),R(x_{n},x_{3}^{p^{2}\beta^{2}})}\\
&=\ldots=\pres{x_{n-1},x_{n}}{x_{n-1}^{\beta^{n-2}\gamma},R(x_{n},x_{n-1}),R(x_{n},x_{n-1}^{p^{n-2}\beta^{n-2}})}.
\end{alignat*}
We set $a=x_{n}$ and $b=x_{n-1}$ to obtain
\[ G_{\Lambda}(R)=\pres{a,b}{b^{\beta^{n-2}\gamma},R(a,b),R(a,b^{p^{n-2}\beta^{n-2}})}. \]
Since $a^\alpha=b^\beta$ in $K$, we can replace each occurrence of $b^\beta$ in the first and third relator by $a^\alpha$. Afterwards, we simplify the third relator, which has become a word in the generator $a$. This yields
\begin{alignat*}{1}
G_{\Lambda}(R)
&=\pres{a,b}{a^{\alpha\beta^{n-3}\gamma},R(a,b),a^{\alpha-\beta(p^{n-2}\alpha\beta^{n-3})}}.
\end{alignat*}
The first and third relator can be subsumed to a single one of the form $a^{r}$, where $r$ is the greatest common divisor of $\alpha\beta^{n-3}\gamma$ and $\alpha-\beta(p^{n-2}\alpha\beta^{n-3})$. That is,
\[ r=(\alpha\beta^{n-3}\gamma,\alpha-\beta(p^{n-2}\alpha\beta^{n-3}))=|\alpha(\gamma,1-p^{n-2}\beta^{n-2})|. \]
In order to evaluate the rightmost term, observe that
\[ p^{n-2}\beta^{n-2}=p^{n-2}(\alpha^{n-2}-\gamma)\equiv p^{n-2}\alpha^{n-2}=(p\alpha)^{n-2}\equiv 1^{n-2}=1~(\text{mod}~\gamma). \]
So $\gamma$ divides $1-p^{n-2}\beta^{n-2}$. Therefore, $r=|\alpha\gamma|=|\alpha(\alpha^{n-2}-\beta^{n-2})|$, which proves the first conclusion. Thus, the abelianization $G_{\Lambda}(R)^{\text{ab}}$ is given by the presentation $\pres{a,b}{R(a,b),a^{\alpha\gamma}}^{\text{ab}}$, whose relation matrix is
\[ A=\left( \begin{array}{cc} \alpha & -\beta \\ \alpha\gamma & 0 \end{array} \right). \]
The diagonal entries of the Smith Normal Form of $A$ are the greatest common divisor $(\alpha,-\beta,\alpha\gamma,0)=1$ and the quotient $|\hspace*{-1pt}\det(A)|/(\alpha,-\beta,\alpha\gamma,0)=|\alpha\beta\gamma|$, whence $G_{\Lambda}(R)^{\mathrm{ab}}\cong\mathbb{Z}_{|\alpha\beta\gamma|}=\mathbb{Z}_{|\alpha\beta(\alpha^{n-2}-\beta^{n-2})|}$.

Finally, by adjoining the relator $a^{\delta_a}$ we see that the group $G_\Lambda(R)=\pres{a,b}{R(a,b),a^{\alpha\gamma}}$ maps onto $\pres{a,b}{a^{\delta_a},b^{\beta}}\cong\Z_{\delta_a}\ast\Z_{|\beta|}$ and by adjoining the relator $b^{\delta_b}$ that it maps onto $\pres{a,b}{a^{\alpha},b^{\delta_b}}\cong\Z_{|\alpha|}\ast\Z_{\delta_b}$. If $G_{\Lambda}(R)$ is finite, then these images must be finite, too, from where the last conclusion follows.
\end{proof}

\begin{lemma}\label{lem:M1Lambda1110or1101}
Let $R(a,b)$ be as in Theorem~\ref{mainthm:balanced}. Further suppose that $(\alpha,\beta)=1$ and $a^\alpha=b^\beta$ in~$K$. Then the following hold:
\begin{itemize}
\item[\brkt{a}] If $\Lambda=\Lambda(n;\rarw{m},\larw{1})$ \brkt{$n\geq 2$, $m\geq 1$}, then \[ G_{\Lambda}(R)\cong K/\ngpres{b^{\beta^{m}(\alpha^n-\beta^n)}}^K\quad\text{and}\quad G_{\Lambda}(R)^{\mathrm{ab}}\cong\Z_{|\alpha\beta^m(\alpha^n-\beta^n)|}. \]
\item[\brkt{b}] If $\Lambda=\Lambda(n;\larw{m},\rarw{1})$ \brkt{$n\geq 2$, $m\geq 1$}, then \[ G_{\Lambda}(R)\cong K/\ngpres{a^{\alpha^{m}(\alpha^n-\beta^n)}}^K\quad\text{and}\quad G_{\Lambda}(R)^{\mathrm{ab}}\cong\Z_{|\alpha^m\beta(\alpha^n-\beta^n)|}. \]
\end{itemize}
If, in addition, $|\alpha|\neq 1$, $|\beta|\neq 1$, and $G_\Lambda(R)$ is finite, then $\delta_a=\delta_b=1$.
\end{lemma}

\begin{proof}
The proof begins essentially the same way as the one of Lemma~\ref{lem:M1Lambda1010or0101}. Again, we prove Part~(a) only. Part~(b) can either be proved similarly or it can be deduced from Part\,(a) using the reflection principle addressed in Remark~\ref{rem:reflection}. The group $G_{\Lambda}(R)$ is defined by the presentation
\[ \hugepres{
\begin{array}{l}
x_1,\ldots,x_n,\\
y_1,\ldots,y_m,z
\end{array}
}{
\begin{array}{l}
R(x_1,x_2),R(x_2,x_3),\ldots,R(x_n,x_1), \\
R(x_n,y_1),R(y_1,y_2),\ldots,R(y_{m-1},y_m),R(z,y_m)
\end{array}
}. \]
We set $\gamma=\alpha^{n}-\beta^{n}$ and apply precisely the same transformations as in the proof of Lemma~\ref{lem:M1Lambda1010or0101}. What remains is $G_{\Lambda}(R)=\pres{y_m,z}{R(z,y_m),y_m^{\beta^{m}\gamma}}$. Next, we set $a=z$ and $b=y_m$ to obtain
\[ G_{\Lambda}(R)=\pres{a,b}{R(a,b),b^{\beta^{m}\gamma}}. \]
The relation matrix of the corresponding presentation of $G_{\Lambda}(R)^{\text{ab}}$ is
\[ A=\left( \begin{array}{cc} \alpha & -\beta \\ 0 & \beta^{m}\gamma \end{array} \right). \]
The diagonal entries of the Smith Normal Form of $A$ are the greatest common divisor $(\alpha,-\beta,0,\beta^m\gamma)=1$ and the quotient $|\hspace*{-1pt}\det(A)|/(\alpha,-\beta,0,\beta^m\gamma)=|\alpha\beta^{m}\gamma|$, whence $G_{\Lambda}(R)^{\mathrm{ab}}\cong\mathbb{Z}_{|\alpha\beta^m\gamma|}=\mathbb{Z}_{|\alpha\beta^m(\alpha^{n}-\beta^{n})|}$.

Finally, by adjoining the relator $a^{\delta_a}$ we see that the group $G_\Lambda(R)=\pres{a,b}{R(a,b), b^{\beta^m\gamma}}$ maps onto $\pres{a,b}{a^{\delta_a},b^{\beta}}\cong\Z_{\delta_a}\ast\Z_{|\beta|}$ and by adjoining the relator $b^{\delta_b}$ that it maps onto $\pres{a,b}{a^{\alpha},b^{\delta_b}}\cong\Z_{|\alpha|}\ast\Z_{\delta_b}$. If $G_{\Lambda}(R)$ is finite, then these images must be finite, too, from where the last conclusion follows.
\end{proof}

\subsection{Cases where \texorpdfstring{$|\alpha|=1$}{|alpha|=1} and \texorpdfstring{$|\beta|\geq 1$}{|beta|>=1}}
\label{subsec:proofcase23}

Let us now suppose that $|\alpha|=1$ and that $a^{\alpha}=b^{\beta}$ in $K$. Because $|\alpha|=1$, the equation $a^{\alpha}=b^{\beta}$ is equivalent to $a=b^{\alpha\beta}$ in $K$ and so the relations $R(a,b)=1$ and $a=b^{\alpha\beta}$ can be derived from each other. Thus, whenever we encounter some relator $R(a,b)$, we may replace it by $a=b^{\alpha\beta}$. Now, consider a digraph $\Lambda$ and the presentation $P_{\Lambda}(R)$. If $i\in V(\Lambda)$ is a source leaf and $(i,j)\in A(\Lambda)$ is the only arc incident with it, then the relation $x_{i}=x_{j}^{\alpha\beta}$ is the only one involving $x_{i}$. We can therefore remove the generator $x_{i}$ together with that relation, leaving the other relations unchanged. In other words, if we \emph{prune} the source leaf $i$, \ie consider the digraph $\Lambda'$ obtained from $\Lambda$ by removing $i\in V(\Lambda)$ and $(i,j)\in A(\Lambda)$, then $G_{\Lambda'}(R)\cong G_{\Lambda}(R)$. Analogously, if $|\beta|=1$, then we can to prune any sink leaf without changing the isomorphism type of the resulting group. We first consider the case where $|\alpha|=1$ and $|\beta|\geq 2$.

\begin{lemma}\label{lem:oneofalphabetaisone}
Let $\Lambda$ and $R(a,b)$ be as in Theorem~\ref{mainthm:balanced}. Further suppose that $\Lambda$ is weakly connected and that $|\alpha|=1$, $|\beta|\geq 2$, and $a^{\alpha}=b^{\beta}$ in $K$. If $G_{\Lambda}(R)$ is finite, then after recursively pruning all source leaves $\Lambda$ becomes one of the following digraphs:
\begin{center}
\begin{tabular}{l@{\hspace{10mm}}l}
\brkt{a} $\Lambda_{s}=\Lambda(n)$, & \brkt{b} $\Lambda_{s}=\Lambda(n;\rarw{m})$ \brkt{$m\geq 1$}, \\
\brkt{c} $\Lambda_{s}=\Lambda(n,d)$ \brkt{$\frac{n}{2}\geq d\geq 1$}, & \brkt{d} $\Lambda_{s}=\Lambda(n,d;\rarw{m})$ \brkt{$\frac{n}{2}\geq d\geq 1$, $m\geq 1$}, \\
\brkt{e} $\Lambda_{s}=\Lambda(n;\larw{m},\rarw{\ell})$ \brkt{$m\geq 1$, $\ell\geq 1$}. &
\end{tabular}
\end{center}
\end{lemma}

\begin{proof}
Suppose that $G_{\Lambda}(R)$ is finite. After recursively pruning all source leaves we obtain a digraph $\Lambda_{s}$ that is still non-empty and finite, has an equal number of vertices and arcs, and is weakly connected. Therefore, the underlying undirected graph of $\Lambda_{s}$ has precisely one cycle. Moreover, as explained above, since $|\alpha|=1$, pruning a source leaf does not change the isomorphism type of the resulting group, whence $G_{\Lambda_{s}}(R)$ is finite, too. Suppose that there are distinct vertices $u,w\in V(\Lambda_{s})$ that are both sinks. Killing all generators $x_v$ ($v \in V(\Lambda_{s})\smallsetminus \{u,w\}$) gives that $G_{\Lambda_{s}}(R)$ maps onto $\Z_{|\beta|}\ast\Z_{|\beta|}$, which is infinite. Thus, we can assume that $\Lambda_{s}$ has at most one sink. As in the proof of Lemma~\ref{lem:M2allowedbalancedgraphs}, let $\sigma$ and $\tau$ be the numbers of sources and sinks, respectively. By construction of $\Lambda_{s}$, there are no source leaves. Let $\tau_1$ be the number of sink leaves (which is the number of all leaves). Then $0\leq \tau_1\leq \tau\leq 1$. If $\tau=0$, then $\tau_{1}=0$ and the underlying undirected graph of $\Lambda_{s}$ is the $n$-cycle. In this situation $\tau=0$ implies that $\sigma=\tau=0$, whence $\Lambda_{s}=\Lambda(n)$. On the other hand, if $\tau=1$, then $\tau_{1}\in\sset{0,1}$. If $\tau_{1}=0$, the same argument shows that $\sigma=\tau=1$ and $\Lambda_{s}=\Lambda(n,d)$ ($\frac{n}{2}\geq d\geq 1$). If $\tau_{1}=1$, the underlying undirected graph of $\Lambda_{s}$ is a \emph{tadpole}, \ie a graph that consists of a cycle $C$ and a path $P$ joined by a bridge $B$, see Figure~\ref{fig:tadpole}. Because the leaf is the only sink of $\Lambda_{s}$, there can be at most one source on $P$ and at most one source on $C$.

\begin{enumerate}
\item If there is no source on $P$, the bridge $B$ must point away from the cycle $C$ and we distinguish between two cases. If there is no source on $C$, then $\Lambda_{s}=\Lambda(n;\rarw{m})$ ($m\geq 1$). If there is a source on~$C$, then $\Lambda_{s}=\Lambda(n,d;\rarw{m})$ ($\frac{n}{2}\geq d\geq 1$, $m\geq 1$).
\item If there is one source on $P$, then the bridge $B$ must point towards the cycle $C$ and the fact that there cannot be a sink on $C$ implies that $\Lambda_{s}=\Lambda(n;\larw{m},\rarw{\ell})$ ($m\geq 1$, $\ell\geq 1$).\hfill\qedhere
\end{enumerate}
\end{proof}

\begin{figure}
\begin{center}
\includegraphics{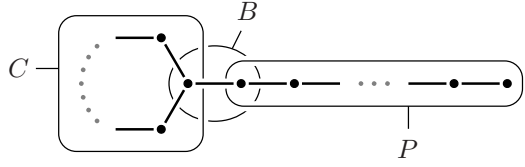}
\begin{picture}(0,0)
\put(-200,31.5){\small $C$}
\put(-54.5,1){\small $P$}
\put(-114.5,53){\small $B$}
\end{picture}
\caption{A tadpole consists of a cycle $C$ and a path $P$ joined by a bridge $B$.}
\label{fig:tadpole}
\end{center}
\end{figure}

The groups $G_{\Lambda_{s}}(R)$ arising in Cases~(a) and (b) of Lemma~\ref{lem:oneofalphabetaisone} are identified in Lemmas~\ref{lem:M1Lambda0000} and~\ref{lem:M1Lambda1010or0101}\,(a), respectively. Lemmas~\ref{lem:M1LambdaS1010}, \ref{lem:M1LambdaS1010tail}, and \ref{lem:M1LambdaS1010twotails} will address Cases~(c), (d), and (e), respectively.

\begin{lemma}\label{lem:M1LambdaS1010}
Let $R(a,b)$ be as in Theorem~\ref{mainthm:balanced}. Further suppose that $|\alpha|=1$, $|\beta|\geq 2$, and $a^\alpha=b^\beta$ in $K$. If $\Lambda_{s}=\Lambda(n,d)$ \brkt{$n\geq 3$, $\frac{n}{2}\geq d\geq 1$}, then $G_{\Lambda_{s}}(R)\cong\Z_{|(\alpha\beta)^{n-d}-(\alpha\beta)^{d}|}$. In particular, $G_{\Lambda_{s}}(R)$ is infinite if and only if $d=\frac{n}{2}$.
\end{lemma}

\begin{proof}
Let $V(\Lambda) = \{1,\ldots,n\}$. Without loss of generality let $n$ be the source and $d$ be the sink. As mentioned at the beginning of Section~\ref{subsec:proofcase23}, we may replace any occurrence of $R(a,b)$ by $a=b^{\alpha\beta}$. So,
\begin{alignat*}{1}
G_{\Lambda_{s}}(R)
&=\hugepres{x_{1},\ldots,x_{n}}{
\begin{array}{l}
x_{n}=x_{1}^{\alpha\beta},x_{1}=x_{2}^{\alpha\beta},\ldots,x_{d-1}=x_{d}^{\alpha\beta},\\
x_{n}=x_{n-1}^{\alpha\beta},x_{n-1}=x_{n-2}^{\alpha\beta},\ldots,x_{d+1}=x_{d}^{\alpha\beta}
\end{array}}\\
&=\pres{x_{d},x_{n}}{x_{n}=x_{d}^{(\alpha\beta)^{d}},x_{n}=x_{d}^{(\alpha\beta)^{n-d}}}\\
&=\pres{x_{d}}{x_{d}^{(\alpha\beta)^{n-d}-(\alpha\beta)^{d}}}.
\end{alignat*}
Therefore, $G_{\Lambda_{s}}(R)\cong\Z_{|(\alpha\beta)^{n-d}-(\alpha\beta)^{d}|}$.
\end{proof}

\begin{lemma}\label{lem:M1LambdaS1010tail}
Let $R(a,b)$ be as in Theorem~\ref{mainthm:balanced}. Further suppose that $|\alpha|=1$, $|\beta|\geq 2$, and $a^\alpha=b^\beta$ in~$K$. If $\Lambda_{s}=\Lambda(n,d;\rarw{m})$ \brkt{$n\geq 3$, $\frac{n}{2}\geq d\geq 1$, $m\geq 1$}, then $G_{\Lambda_{s}}(R)\cong\Z_{|\beta^{m}((\alpha\beta)^{n-d}-(\alpha\beta)^{d})|}$. In particular, $G_{\Lambda_{s}}(R)$ is infinite if and only if $d=\frac{n}{2}$.
\end{lemma}

\begin{proof}
As in the proof of Lemma~\ref{lem:M1LambdaS1010}, we have
\begin{alignat*}{1}
G_{\Lambda_{s}}(R)
&=\hugepres{
\begin{array}{l}
x_1,\ldots,x_n,\\
y_1,\ldots, y_m
\end{array}
}{
\begin{array}{l}
x_{n}=x_{1}^{\alpha\beta},x_{1}=x_{2}^{\alpha\beta},\ldots,x_{d-1}=x_{d}^{\alpha\beta},\\
x_{n}=x_{n-1}^{\alpha\beta},x_{n-1}=x_{n-2}^{\alpha\beta},\ldots,x_{d+1}=x_{d}^{\alpha\beta},\\
x_{d}=y_{1}^{\alpha\beta},y_{1}=y_{2}^{\alpha\beta},\ldots,y_{m-1}=y_{m}^{\alpha\beta}
\end{array}}\\
&=\pres{x_{d},y_{m}}{x_{d}^{(\alpha\beta)^{n-d}-(\alpha\beta)^{d}},x_{d}=y_{m}^{(\alpha\beta)^{m}}}\\
&=\pres{y_{m}}{y_{m}^{(\alpha\beta)^{m}((\alpha\beta)^{n-d}-(\alpha\beta)^{d})}}.
\end{alignat*}
Therefore, $G_{\Lambda_{s}}(R)\cong\Z_{|\beta^{m}((\alpha\beta)^{n-d}-(\alpha\beta)^{d})|}$.
\end{proof}

\begin{lemma}\label{lem:M1LambdaS1010twotails}
Let $R(a,b)$ be as in Theorem~\ref{mainthm:balanced}. Further suppose that $|\alpha|=1$, $|\beta|\geq 2$, and $a^\alpha=b^\beta$ in~$K$. If $\Lambda_{s}=\Lambda(n;\larw{m},\rarw{\ell})$ \brkt{$n\geq 2$, $m\geq 1$, $\ell\geq 1$}, then $G_{\Lambda_{s}}(R)\cong\Z_{|\beta^{\ell}(\alpha^{n}-\beta^{n})|}$.
\end{lemma}

\begin{proof}
We have
\begin{alignat*}{1}
G_{\Lambda_{s}}(R)
&=\hugepres{
\begin{array}{l}
x_1,\ldots,x_n,\\
y_1,\ldots, y_m,\\
z_1,\ldots,z_\ell
\end{array}
}{
\begin{array}{l}
x_{1}=x_{2}^{\alpha\beta},x_{2}=x_{3}^{\alpha\beta},\ldots,x_{n-1}=x_{n}^{\alpha\beta},x_{n}=x_{1}^{\alpha\beta},\\
y_{m}=y_{m-1}^{\alpha\beta},y_{m-1}=y_{m-2}^{\alpha\beta},\ldots,y_{2}=y_{1}^{\alpha\beta},y_{1}=x_{n}^{\alpha\beta},\\
y_{m}=z_{1}^{\alpha\beta},z_{1}=z_{2}^{\alpha\beta},\ldots,z_{\ell-1}=z_{\ell}^{\alpha\beta}
\end{array}}\\
&=\pres{x_{n},y_{m},z_{\ell}}{x_{n}=x_{n}^{(\alpha\beta)^{n}},y_{m}=x_{n}^{(\alpha\beta)^{m}},y_{m}=z_{\ell}^{(\alpha\beta)^{\ell}}}\\
&=\pres{x_{n},y_{m},z_{\ell}}{x_{n}^{\alpha^{n}}=x_{n}^{\beta^{n}},y_{m}=x_{n}^{(\alpha\beta)^{m}},y_{m}=z_{\ell}^{(\alpha\beta)^{\ell}}}\\
&=\pres{x_{n},y_{m},z_{\ell}}{x_{n}^{\alpha^{n}-\beta^{n}},y_{m}=x_{n}^{(\alpha\beta)^{m}},y_{m}=z_{\ell}^{(\alpha\beta)^{\ell}}}.
\end{alignat*}
We set $x=x_{n}$, $y=y_{m}$, $z=z_{\ell}$, and $\gamma=\alpha^{n}-\beta^{n}$, $s=(\alpha\beta)^{m}$. Since $(\gamma,s)=1$, there is an integer $t\in\mathbb{Z}$ such that $st\equiv 1~(\text{mod}~\gamma)$. Thus,
\begin{alignat*}{1}
G_{\Lambda_{s}}(R)
&=\pres{x,y,z}{x^{\gamma},y=x^{s},y=z^{(\alpha\beta)^{\ell}}}\\
&=\pres{x,y,z}{x^{\gamma},y=x^{s},y=z^{(\alpha\beta)^{\ell}},y^{\gamma},y^{t}=x^{st}}\\
&=\pres{x,y,z}{x^{\gamma},y=x^{s},y=z^{(\alpha\beta)^{\ell}},y^{\gamma},y^{t}=x}\\
&=\pres{y,z}{y^{\gamma t},y=y^{st},y=z^{(\alpha\beta)^{\ell}},y^{\gamma}}\\
&=\pres{y,z}{y=z^{(\alpha\beta)^{\ell}},y^{\gamma}}=\pres{z}{z^{(\alpha\beta)^{\ell}\gamma}}.
\end{alignat*}
Therefore, $G_{\Lambda_{s}}(R)\cong\Z_{|\beta^{\ell}\gamma|}=\Z_{|\beta^{\ell}(\alpha^{n}-\beta^{n})|}$.
\end{proof}

We now turn to the case where $|\alpha|=|\beta|=1$.

\begin{lemma}\label{lem:bothalphabetaareone}
Let $\Lambda$ and $R(a,b)$ be as in Theorem~\ref{mainthm:balanced}. Further suppose that $\Lambda$ is weakly connected and that $|\alpha|=|\beta|=1$ and $a^{\alpha}=b^{\beta}$ in $K$. If $G_{\Lambda}(R)$ is finite, then $\alpha=-\beta$, $n$ is odd, and $G_{\Lambda}(R)\cong\Z_{2}$.
\end{lemma}

\begin{proof}
Suppose that $G_{\Lambda}(R)$ is finite. After recursively pruning all source and all sink leaves we obtain a digraph $\Lambda_{st}$ whose underlying undirected graph is the $n$-cycle. As explained at the beginning of Section~\ref{subsec:proofcase23}, since $|\alpha|=|\beta|=1$, pruning any leaf does not change the isomorphism type of the resulting group. Thus, if $\alpha=\beta$, then
\[
G_{\Lambda}(R)=\pres{x_{1},\ldots,x_{n}}{x_{1}=x_{2},x_{2}=x_{3},\ldots,x_{n}=x_{1}}\cong\mathbb{Z}.
\]
So assume that $\alpha=-\beta$, in which case
\[
G_{\Lambda}(R)=\pres{x_{1},\ldots,x_{n}}{x_{1}=x_{2}^{-1},x_{2}=x_{3}^{-1},\ldots,x_{n}=x_{1}^{-1}}\cong\left\{\begin{array}{ll} \mathbb{Z} & \text{if $n$ is even,} \\ \mathbb{Z}_{2} & \text{if $n$ is odd.} \end{array}\right.
\]
\end{proof}

\subsection{Proof of Theorem~\ref{mainthm:balanced}}
\label{subsec:prooftheorema}

We first restrict ourselves to the case where the digraph $\Lambda$ is weakly connected. The conclusion that $G_{\Lambda}(R)$ is non-trivial will then allow us to deal with the case where $\Lambda$ is not weakly connected. So suppose that $\Lambda$ is weakly connected and that $G_{\Lambda}(R)$ is finite. By Corollary~\ref{cor:trianglefreeimpliesinfinite}, we can assume that $K$ does not have Property~$W_1$. Then, by Proposition~\ref{prop:center}, we have $\alpha\neq 0$, $\beta\neq 0$, and $a^\alpha=b^\beta$ in $K$. If $|\alpha|\geq 2$ and $|\beta|\geq 2$, then Lemma~\ref{lem:M2allowedbalancedgraphs} implies that $(\alpha,\beta)=1$, whence $\alpha^{n}-\beta^{n}\neq 0$, and that $\Lambda$ is one of the following digraphs:
\begin{center}
\begin{tabular}{l@{\hspace{10mm}}l@{\hspace{10mm}}l}
\brkt{a} $\Lambda(n)$, & \brkt{b} $\Lambda(n;\rarw{m})$ \brkt{$m\geq 1$}, & \brkt{c} $\Lambda(n;\larw{m})$ \brkt{$m\geq 1$}, \\
\brkt{d} $\Lambda(n,1)$, & \brkt{e} $\Lambda(n;\rarw{m},\larw{1})$ \brkt{$m\geq 1$}, & \brkt{f} $\Lambda(n;\larw{m},\rarw{1})$ \brkt{$m\geq 1$}.
\end{tabular}
\end{center}
The result then follows from Lemmas~\ref{lem:M1Lambda0000}, \ref{lem:M1Lambda1010or0101}, \ref{lem:M1Lambda1100}, \ref{lem:M1Lambda1110or1101}. Note that in each case $G_{\Lambda}(R)$ is non-trivial. Indeed, in Cases (a)--(c) this can be deduced directly from the lemmas. In Cases (d)--(f) observe that the abelianization $G_{\Lambda}(R)^{\mathrm{ab}}$ is non-trivial, and so is $G_{\Lambda}(R)$. If $|\alpha|=1$ and $|\beta|\geq 2$, then $(\alpha,\beta)=1$ and $\alpha^{n}-\beta^{n}\neq 0$ are obviously satisfied.\begin{samepage} Lemma~\ref{lem:oneofalphabetaisone} implies that after recursively pruning all source leaves $\Lambda$ becomes one of the following digraphs:
\begin{center}
\begin{tabular}{l@{\hspace{10mm}}l}
\brkt{a} $\Lambda_{s}=\Lambda(n)$, & \brkt{b} $\Lambda_{s}=\Lambda(n;\rarw{m})$ \brkt{$m\geq 1$}, \\
\brkt{c} $\Lambda_{s}=\Lambda(n,d)$ \brkt{$\frac{n}{2}\geq d\geq 1$}, & \brkt{d} $\Lambda_{s}=\Lambda(n,d;\rarw{m})$ \brkt{$\frac{n}{2}\geq d\geq 1$, $m\geq 1$}, \\
\brkt{e} $\Lambda_{s}=\Lambda(n;\larw{m},\rarw{\ell})$ \brkt{$m\geq 1$, $\ell\geq 1$}. &
\end{tabular}
\end{center}
\end{samepage}
The result then follows from Lemmas~\ref{lem:M1Lambda0000}, \ref{lem:M1Lambda1010or0101}\,(a), \ref{lem:M1LambdaS1010}, \ref{lem:M1LambdaS1010tail}, \ref{lem:M1LambdaS1010twotails}. Again, it turns out that in each case $G_{\Lambda}(R)$ is non-trivial. If $|\alpha|\geq 2$ and $|\beta|=1$, then we have a mere reflection of the previous situation. More precisely, due to the reflection principle addressed in Remark~\ref{rem:reflection}, there is a one-to-one correspondence between the two situations: starting off from a digraph $\Lambda$ and a word $R(a,b)$ with $|\alpha|\geq 2$ and $|\beta|=1$, we reverse the direction of each arc, interchange $a$ and $b$, and further replace every letter by its inverse so that also $\alpha$ and $\beta$ are interchanged (without any change of sign), to obtain a digraph $\Lambda'$ and a word $R\hspace{1pt}'(a,b)$ with $|\alpha|=1$ and $|\beta|\geq 2$. Since the resulting groups are isomorphic, we can translate the classification of finite groups from one situation to the other. Note that we now have to consider the digraph $\Lambda_{t}$ obtained from $\Lambda$ by recursively pruning all sink leaves.

If $|\alpha|=|\beta|=1$, then $(\alpha,\beta)=1$ is obviously satisfied. Moreover, Lemma~\ref{lem:bothalphabetaareone} implies that $\alpha=-\beta$ and $n$~is odd, whence $|\alpha^{n}-\beta^{n}|=2$ and, in particular, $\alpha^{n}-\beta^{n}\neq 0$. The lemma further implies that $G_{\Lambda}(R)\cong\Z_{2}$, which is non-trivial. This completes the proof for the case where the digraph $\Lambda$ is weakly connected and, in particular, shows that in this case $G_{\Lambda}(R)$ is non-trivial.

What remains is to consider the case where $\Lambda$ is not weakly connected, for which we claim that $G_{\Lambda}(R)$ cannot be finite. Indeed, if $\Lambda$ has weakly connected components $\Lambda_1,\ldots,\Lambda_k$ with $k\geq 2$, then $G_{\Lambda}(R)$ is isomorphic to the free product $G_{\Lambda_1}(R)\ast\ldots\ast G_{\Lambda_k}(R)$. If for some $1 \leq i\leq k$ the weakly connected component $\Lambda_i$ has more vertices than arcs, then $P_{\Lambda_i}(R)$ has more generators than relators. Hence, $G_{\Lambda_i}(R)$ is infinite, and so is $G_{\Lambda}(R)$. Because of this we can assume that each $\Lambda_i$ has at most as many vertices as arcs. If any $\Lambda_i$ has fewer vertices than arcs, then $\Lambda$ has fewer vertices than arcs, which contradicts the hypothesis of Theorem~\ref{mainthm:balanced}. Therefore, each $\Lambda_i$ has an equal number of vertices and arcs. But then, by the above, each $G_{\Lambda_i}(R)$ is non-trivial and thus $G_{\Lambda}(R)\cong G_{\Lambda_1}(R)\ast\ldots\ast G_{\Lambda_k}(R)$ is infinite.

\bibliographystyle{plain}

\begin{thebibliography}{10}

\bibitem{Allcock}
D.~{Allcock}, {Triangles of Baumslag--Solitar groups}, {Can. J. Math.} 64
  (2012) 241--253.
\newblock \href {http://dx.doi.org/10.4153/CJM-2011-062-8}
  {\path{doi:10.4153/CJM-2011-062-8}}.

\bibitem{AndrewsCurtis}
J.~J. {Andrews}, M.~L. {Curtis}, {Free groups and handlebodies}, {Proc. Am.
  Math. Soc.} 16 (1965) 192--195.
\newblock \href {http://dx.doi.org/10.2307/2033843}
  {\path{doi:10.2307/2033843}}.

\bibitem{BogleyWilliams}
W.~A. {Bogley}, G.~{Williams}, {Efficient finite groups arising in the study of
  relative asphericity}, {Math. Z.} 284 (2016) 507--535.
\newblock \href {http://dx.doi.org/10.1007/s00209-016-1664-3}
  {\path{doi:10.1007/s00209-016-1664-3}}.

\bibitem{CampbellRobertson}
C.~M. {Campbell}, E.~F. {Robertson}, {On a group presentation due to Fox},
  {Can. Math. Bull.} 19 (1976) 247--248.
\newblock \href {http://dx.doi.org/10.4153/CMB-1976-039-9}
  {\path{doi:10.4153/CMB-1976-039-9}}.

\bibitem{Collins04}
D.~J. {Collins}, {Intersections of Magnus subgroups of one-relator groups}, in:
  T.~W. {M\"uller} (Ed.), {Groups: Topological, combinatorial and arithmetic
  aspects}, Vol. 311 of {Lond. Math. Soc. Lect. Note Ser.}, Cambridge
  University Press, Cambridge, 2004, pp. 255--296.
\newblock \href {http://dx.doi.org/10.1017/cbo9780511550706.009}
  {\path{doi:10.1017/cbo9780511550706.009}}.

\bibitem{Collins08}
D.~J. {Collins}, {Intersections of conjugates of Magnus subgroups of
  one-relator groups}, in: M.~{Boileau}, M.~{Scharlemann}, R.~{Weidmann}
  (Eds.), {The Zieschang Gedenkschrift}, Vol.~14 of {Geometry \& Topology
  Monographs}, Geometry \& Topology Publications, Coventry, 2008, pp. 135--171.
\newblock \href {http://dx.doi.org/10.2140/gtm.2008.14.135}
  {\path{doi:10.2140/gtm.2008.14.135}}.

\bibitem{EdjvetHowie08}
M.~{Edjvet}, J.~{Howie}, {Intersections of Magnus subgroups and embedding
  theorems for cyclically presented groups}, {J. Pure Appl. Algebra} 212 (2008)
  47--52.
\newblock \href {http://dx.doi.org/10.1016/j.jpaa.2007.04.009}
  {\path{doi:10.1016/j.jpaa.2007.04.009}}.

\bibitem{GolodShafarevich}
E.~S. {Golod}, I.~R. {Shafarevich}, {On the class field tower}, {Izv. Akad.
  Nauk SSSR Ser. Mat.} 28 (1964) 261--272, {E}nglish translation in {Am. Math.
  Soc. Transl. Ser. 2} 48 (1965) 91--102.
\newblock \href {http://dx.doi.org/10.1090/trans2/048/05}
  {\path{doi:10.1090/trans2/048/05}}.

\bibitem{HavasNewmanOBrien}
G.~{Havas}, M.~F. {Newman}, E.~A. {O'Brien}, {Groups of deficiency zero}, in:
  G.~{Baumslag}, D.~{Epstein}, R.~{Gilman}, H.~{Short}, C.~{Sims} (Eds.),
  {Geometric and computational perspectives on infinite groups}, Vol.~25 of
  DIMACS Series in Discrete Mathematics and Theoretical Computer Science,
  American Mathematical Society, Providence, RI, 1996, pp. 53--67.
\newblock \href {http://dx.doi.org/10.1090/dimacs/025/04}
  {\path{doi:10.1090/dimacs/025/04}}.

\bibitem{Higman}
G.~{Higman}, {A finitely generated infinite simple group}, {J. Lond. Math.
  Soc.} 26 (1951) 61--64.
\newblock \href {http://dx.doi.org/10.1112/jlms/s1-26.1.61}
  {\path{doi:10.1112/jlms/s1-26.1.61}}.

\bibitem{Howie05}
J.~{Howie}, {Magnus intersections in one-relator products}, {Mich. Math. J.} 53
  (2005) 597--623.
\newblock \href {http://dx.doi.org/10.1307/mmj/1133894169}
  {\path{doi:10.1307/mmj/1133894169}}.

\bibitem{Johnson79}
D.~L. {Johnson}, {A new class of 3-generator finite groups of deficiency zero},
  {J. Lond. Math. Soc. (2)} 19 (1979) 59--61.
\newblock \href {http://dx.doi.org/10.1112/jlms/s2-19.1.59}
  {\path{doi:10.1112/jlms/s2-19.1.59}}.

\bibitem{JohnsonBook}
D.~L. {Johnson}, {Topics in the theory of group presentations}, Vol.~42 of
  {Lond. Math. Soc. Lect. Note Ser.}, Cambridge University Press, Cambridge,
  1980.
\newblock \href {http://dx.doi.org/10.1017/cbo9780511629303}
  {\path{doi:10.1017/cbo9780511629303}}.

\bibitem{JohnsonRobertson}
D.~L. {Johnson}, E.~F. {Robertson}, {Finite groups of deficiency zero}, in:
  C.~T.~C. {Wall} (Ed.), {Homological group theory}, Vol.~36 of {Lond. Math.
  Soc. Lect. Note Ser.}, {Cambridge University Press, Cambridge}, 1979, pp.
  275--290.
\newblock \href {http://dx.doi.org/10.1017/cbo9781107325449.018}
  {\path{doi:10.1017/cbo9781107325449.018}}.

\bibitem{JohnsonWamsleyWright}
D.~L. {Johnson}, J.~W. {Wamsley}, D.~{Wright}, {The Fibonacci groups}, {Proc.
  Lond. Math. Soc. (3)} 29 (1974) 577--592.
\newblock \href {http://dx.doi.org/10.1112/plms/s3-29.4.577}
  {\path{doi:10.1112/plms/s3-29.4.577}}.

\bibitem{Magnus}
W.~{Magnus}, {\"Uber diskontinuierliche Gruppen mit einer definierenden
  Relation (Der Freiheitssatz)}, J. Reine Angew. Math. 163 (1930) 141--165.
\newblock \href {http://dx.doi.org/10.1515/crll.1930.163.141}
  {\path{doi:10.1515/crll.1930.163.141}}.

\bibitem{Mennicke}
J.~{Mennicke}, {Einige endliche Gruppen mit drei Erzeugenden und drei
  Relationen}, {Arch. Math. (Basel)} 10 (1959) 409--418.
\newblock \href {http://dx.doi.org/10.1007/BF01240820}
  {\path{doi:10.1007/BF01240820}}.

\bibitem{Pride}
S.~J. {Pride}, {Groups with presentations in which each defining relator
  involves exactly two generators}, {J. Lond. Math. Soc. (2)} 36 (1987)
  245--256.
\newblock \href {http://dx.doi.org/10.1112/jlms/s2-36.2.245}
  {\path{doi:10.1112/jlms/s2-36.2.245}}.

\bibitem{Stallings}
J.~R. {Stallings}, {Non-positively curved triangles of groups}, in: E.~{Ghys},
  A.~{Haefliger}, A.~{Verjowski} (Eds.), {Group theory from a geometrical
  viewpoint}, World Scientific, Singapore, 1991, pp. 491--503.
\newblock \href {http://dx.doi.org/10.1142/1235} {\path{doi:10.1142/1235}}.

\end{thebibliography}

\bigskip\centerline{---------------------------------------}\bigskip

\noindent
\textsc{Department of Mathematics, Ludwigsburg University of Education, \\ Reuteallee 46, 71634 Ludwigsburg, Germany.}\par\nopagebreak
\textit{E-mail address:} \texttt{johannes.cuno@ph-ludwigsburg.de}

\medskip

\noindent
\textsc{Department of Mathematical Sciences, University of Essex, \\ Wivenhoe Park, Colchester, Essex, CO4 3SQ, UK.}\par\nopagebreak
\textit{E-mail address:} \texttt{gerald.williams@essex.ac.uk}

\end{document}